\documentclass[hidelinks,12pt]{article}
\def\currentversion{arxiv} 
\def\compareversion{npjqi}

\usepackage[T1]{fontenc}
\usepackage[utf8]{inputenc} 
\usepackage{graphicx} 
\usepackage{amssymb}
\usepackage{amsmath}
\usepackage{amsthm}
\usepackage{doi}
\newtheorem{theorem}{Theorem}[]
\newtheorem{proposition}[theorem]{Proposition}
\newtheorem{lemma}[theorem]{Lemma}
\newtheorem{remark}{Remark}[]

\newtheorem{thmA}{Theorem}

\usepackage{soul}

\usepackage{graphicx} %
\usepackage[left=25mm,right=25mm,top=25mm,bottom=25mm]{geometry}

\makeatletter
\let\saved@setaddresses\@setaddresses %
\let\@setaddresses\relax              %
\makeatother   

\usepackage[baseline]{euflag}
\usepackage[shortlabels]{enumitem}

\newcommand{\supp}{\mathop{\mathrm{supp}}\nolimits}
\newcommand{\C}{\mathbb{C}}
\newcommand{\R}{\mathbb{R}}
\newcommand{\Q}{\mathbb{Q}}
\newcommand{\N}{\mathbb{N}}
\newcommand{\GL}[1]{\mathrm{GL}(#1)}
\newcommand{\Tr}{\mathrm{Tr}}
\newcommand{\Span}{\mathrm{Span}}
\newcommand{\rank}{\mathrm{rank}}

\DeclareMathOperator{\Hom}{Hom}
\DeclareMathOperator{\re}{re}
\DeclareMathOperator{\im}{im}

\newcommand{\igor}[1]{{\color{orange}$\big[\![$~\raisebox{.7pt}{I}\!\!\:\raisebox{-.7pt}{K}: \textit{#1}~$]\!\big]$}}

\newcommand{\nando}[1]{{\color{blue}$\big[\![$~\raisebox{.7pt}{N}\!\!\:\raisebox{-.7pt}{L}: \textit{#1}~$]\!\big]$}}

\title{Robust self-testing with CHSH mod 3}
\author{Igor Klep\thanks{University of Ljubljana, Faculty of Mathematics and Physics, Jadranska 21, 1000 Ljubljana \&
University of Primorska, Faculty of Mathematics, Natural Sciences and Information Technologies,
Glagoljaška 8, 6000 Koper, Slovenia. Email: \texttt{igor.klep@fmf.uni-lj.si}} \and Nando Leijenhorst\thanks{Universit\'{e} de Toulouse; LAAS-CNRS, 7 avenue du colonel Roche, F-31400 Toulouse, France. Email: \texttt{nando.leijenhorst@laas.fr}} \and Victor Magron\thanks{Universit\'{e} de Toulouse; LAAS-CNRS, 7 avenue du colonel Roche, F-31400 Toulouse, France. Email: \texttt{victor.magron@laas.fr}}}
\date{\today}

\begin{document}

\maketitle

\begin{abstract}
The CHSH mod~3 Bell inequality is a natural testbed for higher-dimensional quantum nonlocality, yet its maximal quantum violation and self-testing properties have remained unresolved. We determine its exact maximal quantum value and show that, up to unitary equivalence and the natural symmetries of the inequality, it admits a unique optimal irreducible strategy; equivalently, there are four symmetry-related optimal irreducible strategies. Each of these strategies uses a maximally entangled two-qutrit state. We further prove that any strategy whose value is within $\varepsilon$ of the optimum is $O(\sqrt{\varepsilon})$-close, up to local isometries, to a direct sum of 
optimal irreducible strategies.
\end{abstract}
\ifx\compareversion\currentversion
\section*{Introduction}
\else
\section{Introduction}
\fi
Self-testing is a central concept in device-independent quantum information processing, enabling the certification of quantum states and measurements solely from observed correlations. 
It provides one with a powerful primitive for tasks such as verified quantum computation \cite{reichardt2013classical} and randomness expansion \cite{mayers04}. 
A self-testing protocol in quantum mechanics is a way to verify that a set of measurements and/or a state are (equivalent to) a specific set of measurements and/or a specific state. 
For example, certain measurements $A_i$ and states $\psi$ admit unique correlations $\psi^* A_i\psi$, thus discovering that a set of measurements $\tilde A_i$ and a state $\tilde \psi$ admit the same correlations implies that $(\{\tilde A_i\}, \tilde \psi)$ is equivalent to $(\{A_i\}, \psi)$. Here, equivalence is meant up to `trivial' operations that transform a set of operators and a state while keeping the correlations the same, such as unitary transformations or extending the space by an auxiliary Hilbert space where the operators act as identity operators. 

Self-testing is also possible using Bell inequalities. 
A Bell inequality is an inequality in the correlations of two systems that cannot be violated in classical mechanics, but can be violated in quantum mechanics. 
Introduced in \cite{bell1964}, such inequalities have played a central role in experimentally testing quantum theory. 
Their violation certifies the presence of entanglement and demonstrates that the observed correlations cannot be explained by locally causal classical models. 
If a Bell inequality has a unique set of measurements and state that maximize the violation, it can be used for self-testing. 
The most basic and extensively analyzed Bell inequality was introduced by Clauser, Horne, Shimony, and Holt (CHSH) in \cite{clauser1969proposed}. 
In the CHSH setup, two separate devices are considered, each with two possible measurement settings and two possible outcomes. 
It is well established that this inequality reaches its maximal violation when performing maximally incompatible measurements on each qubit of a maximally entangled two-qubit state. 
Numerous extensions of the CHSH inequality have also been proposed for Bell scenarios involving measurements with $d$ possible outcomes. 
The CHSH mod $d$ Bell inequality, introduced in \cite{buhrman_causality_2005}, is a generalization of the famous CHSH inequality, where the measurement settings and outcomes are no longer binary but take values from the set $\{0,1,\dots,d-1\}$ for some integer $d$, and the winning condition is evaluated modulo $d$. 
Although this functional represents a seemingly natural extension of the CHSH inequality, it proves to be surprisingly difficult to analyze. Buhrman and Massar prove in \cite{buhrman_causality_2005} the upper bound 
\[
\frac{1}{d} + \frac{d-1}{d\sqrt{d}}
\]
on the maximal value of the Bell function that can be reached by quantum strategies.
This is the best possible bound for $d=2$ (the standard CHSH inequality), but does not seem sharp for $d > 2$. For $d=3$, Ji et al. \cite{ji_multi-setting_2008} propose a strategy with value
\[
\frac{1}{3} + \frac{2\cos(\pi/18)}{3\sqrt{3}},
\]
and Liang, Lim and Deng \cite{liang_reexamination_2009} give a matching numerical upper bound. However, until now no proof of the exact maximal quantum value was available.
The authors of \cite{kaniewski2019maximal} adapted the CHSH mod $d$ inequality to derive the first analytical self-testing result that does not depend on self-testing for two-dimensional systems. 
A partial self-testing result for the maximally entangled state of two qutrits was established through numerical methods using a different Bell inequality \cite{salavrakos2017bell}. 

In this paper, we investigate whether the CHSH mod 3 Bell inequality can be used for self-testing.
This differs from approaches that design a protocol or Bell inequality specifically to self-test a particular state, e.g., the SATWAP inequality proposed in \cite{salavrakos2017bell}, or the ones proposed in \cite{bamps_sum-of-squares_2015, meyer_robustly_2025}. 

In practice, one can never measure the correlations or the maximal violation of a Bell inequality exactly. It is therefore natural to consider robust self-testing. Informally, a self-test is robust if a measured value close to the optimum (in case of the maximal violation)  implies that the set of measurements and the state is close to a set of measurements and state corresponding to the maximal violation.  
In \cite{manvcinska2024constant}, the authors obtained such a robust self-testing statement for maximally entangled states based on four binary measurements. 
This result is derived by reformulating the robust self-testing method based on the Gowers–Hatami group-theoretic approach \cite{gowers_inverse_2017} into an adequate algebraic framework. 
As in \cite{manvcinska2024constant}, we will leverage this group-theoretic approach to prove a robust self-testing statement for CHSH mod 3. 
We refer to \cite{supic_self-testing_2020} for a review of (robust) self-testing. 

To find an upper bound on the maximal violation of a Bell inequality, one can use the (dual of the) Navascu\'es-Pironio-Ac\'in (NPA) hierarchy \cite{navascues_convergent_2008}, the noncommutative analog of Lasserre's moment-SOS hierarchy \cite{lasserre2001global}, that uses sum-of-Hermitian-squares polynomials \cite{burgdorf_optimization_2016}. Each level of the hierarchy corresponds to a semidefinite program, and an exact feasible solution certifies an upper bound on the maximum violation. Higher levels give better bounds but are more difficult to compute, and the hierarchy converges to the maximal violation when the level $n \to \infty$. In certain cases, the hierarchy admits finite convergence, i.e., there is a finite $n$ such that the $n$-th level gives the maximal violation.  
However, there are also cases that do not have finite convergence (see, e.g., \cite{fanizza_npa_2025}) as a consequence of recently established quantum complexity results and the refutation of Connes’ embedding conjecture \cite{ji2021mip}.

To use sum-of-squares certificates for self-testing proofs, 
one needs an exact optimal solution to the corresponding semidefinite program.
This means that self-testing with Bell inequalities has only been done using Bell inequalities for which it is possible to find an analytic expression of a sum-of-squares certificate, possibly by identifying numbers in a numerical certificate. This leaves many open cases for which a numerical certificate is known, with or without matching constructions of strategies, but where it is not known whether there is a unique optimal strategy (see, e.g., \cite[Section~6]{hrga_certifying_2024} for a list of cases with numerical optimality).

Our first contribution is to show that the rounding method of \cite{cohn_optimality_2024} can be used to overcome this. This rounding method can round a high-precision solution to an exact optimal solution of a (real) semidefinite program, provided there is an exact optimal solution over a number field of low algebraic degree. The rounding method returns a decomposition $Z = T \hat Z T^{\sf T}$ of the positive semidefinite matrix variable in the semidefinite program, where $\hat Z$ is positive definite.

Our second contribution is to observe that self-testing results can already be derived using the rectangular matrix $T$, which is typically much simpler than the matrices $Z$ and $\hat Z$. In particular, it is not necessary to give an exact factorization of $Z$ or $\hat Z$, 
and hence not necessary to write down the exact polynomials appearing in the sum-of-squares certificate.

\if{
Let $p \in \C\langle X, Y\rangle$ be a non-commutative polynomial in variables $X = (X_1, \dots, X_k)$ and $Y = (Y_1, \dots, Y_l)$, and consider (projection-valued) measurements $\{A_i\}_{i=1}^k$ and $\{B_j\}_{j=1}^l$ on separable Hilbert spaces $\mathcal H_A$ and $\mathcal H_B$ respectively, and a state $\psi \in \mathcal H_A \otimes \mathcal H_B$. The inequality 
\[ 
\beta(A, B, \psi) = \psi^* p(A\otimes I_B, I_A \otimes B) \psi \leq \beta_c,
\]
 where $\beta_c$ is the maximum value of $\beta(A, B, \psi)$ that can be obtained through a classical strategy (that is, $\psi = \psi_A \otimes \psi_B$ with $\psi_A \in \mathcal H_A$ and $\psi_B \in \mathcal H_B$), is called a Bell inequality. We denote the maximum value that can be obtained in quantum mechanics by $\beta_q$, and we call $(\{A_i\}_i, \{B_j\}_j, \psi)$ a strategy for the polynomial $p$, or simply a strategy when the polynomial is clear from the context.

Now, let $\mathcal{I}$ be the ideal of universal relations satisfied by all feasible measurement operators $A, B$. 
Suppose $g_1, \dots, g_N\in \C\langle X, Y\rangle$ are such that $p = \lambda -  \sum_j g_j^*g_j + q$ for some $\lambda \in \R$ and $q \in \mathcal I$, then
\begin{equation}\label{eq:poly_upperbound}
 \psi^* p(A, B) \psi = \lambda - \sum_j (g_j(A, B)\psi)^* g_j(A, B)\psi \leq \lambda
\end{equation}
for all strategies $(A, B, \psi)$. Thus $\lambda$ is an upper bound on the maximum value of the Bell function $p$. This is the basis of non-commutative polynomial optimization. See \cite{burgdorf_optimization_2016} for a thorough introduction. 
Such $\lambda$, $q$ and $g_j$ can be found using semidefinite programming \cite{vandenberghe1996semidefinite}. 
Indeed, any sum-of-squares polynomial can be written as $v^*Zv$, where $Z$ is Hermitian positive semidefinite ($Z \succeq 0$), and $v$ is a so-called \emph{border vector} of which the entries form a basis of the non-commutative polynomials of degree at most the maximum degree of $g_j$.\igor{beware dependence on $q$??}  \nando{It is still right though? The degree needed for the border vector only depends on the $g_j$ polynomials, even though $\sum_j g_j^*g_j + q$ might be of different degree}
The explicit semidefinite program can then be written as 
\begin{equation}\label{eq:general_sdp}
    \begin{aligned}
        & \text{inf} && \lambda \\
        & \text{s.t.} &&  \lambda - p
        = v^* Z v \mod \mathcal I,\\
        &&& Z \succeq 0.
    \end{aligned}    
\end{equation}
Solving such a semidefinite program gives a numerical solution, and one can generally find a rational sum-of-squares polynomial with a slightly worse $\lambda$ by relying on a so-called \emph{rounding and projection} algorithm. 
The initial rounding and projection algorithm has been applied for unconstrained polynomial optimization in \cite{parrilo}. 
Noncommutative extensions have been provided in \cite{cafuta2015rational,naceur2025certified}. 

By fixing the entries of the border vector $v$, this gives a finite semidefinite program. The idea of the NPA hierarchy is to increase the maximum degree step by step to get better bounds: the $n$-th level of the hierarchy sets $v = v_n$ to be the vector whose entries form a basis of the space of polynomials of degree at most $n$, and thus takes into account sum-of-squares polynomials of degree at most $2n$. Higher levels give better bounds but are more difficult to compute, and the hierarchy converges to the maximal violation when $n \to \infty$. In certain cases, the hierarchy admits finite convergence, i.e., there is a finite $n$ such that the $n$-th level gives the maximal violation.  
However, there are also cases that do not have finite convergence (see, e.g., \cite{fanizza_npa_2025}) as a consequence of recently established quantum complexity results and the refutation of Connes’ embedding conjecture \cite{ji2021mip}.  

Now suppose that we have a sum-of-squares certificate with $\lambda = \beta_q$, and $(\{A_i\}_i, \{B_j\}_j, \psi)$ is an optimal strategy for $p$, i.e., a strategy with $\beta(A, B, \psi) = \beta_q$. Then the inequality in \eqref{eq:poly_upperbound} is actually an equality, and thus $g_j(A, B) \psi = 0$ for all $j$. Such equalities are used in self-testing proofs to show that the strategy behaves in a certain way, which generally can be used to classify the optimal strategies (see, e.g., \cite{cui_generalization_2020}). For this, it is imperative that the sum-of-squares certificate is both exact and optimal. This means that self-testing with Bell inequalities has only been done using Bell inequalities for which it is possible to find an analytic expression of an Hermitian sum-of-squares certificate, possibly by identifying the numbers in a numerical certificate. This leaves many open cases for which a numerical certificate is known, with or without matching constructions of strategies, but where it is not known whether there is a unique optimal strategy (see, e.g., \cite[Section~6]{hrga_certifying_2024} for a list of cases with numerical optimality).

In this paper, we propose to use the rounding method of \cite{cohn_optimality_2024} to remedy this. This rounding method can find exact optimal solutions to (real) semidefinite programs, provided there is an exact solution over a number field of low algebraic degree. The rounding procedure returns an optimal solution of the form
\[
Z = T \hat Z T^{\sf T},
\]
where $\hat Z$ is positive definite and $T$ is rectangular. 
Note that to find exact polynomials $g_j$ that form a sum-of-squares certificate, it would be required to factor $Z$. Since $\hat Z$ is positive definite, such a factorization exists, but for certification of the upper bound on $\beta_q$ such a factorization is not required.

The main downside of the rounding procedure is that it requires a highly accurate solution to the semidefinite program, which in turn requires solving the semidefinite program with high-precision arithmetic. This greatly constrains the maximum size of the semidefinite program, so for the Bell inequality in this paper we use symmetry reduction and a particular choice of the border vector $v$  to reduce the size. 

Our second contribution is to note that this factorization also suffices for proving self-testing statements. Since each $g_j$ is a linear combination of the polynomials in the vector $T^{\sf T} v$, we have $g_j(A, B)\psi = 0$ for every $j$ if and only if $T_i^{\sf T}v(A,B)\psi = 0$ for every $i$, $T_i$ is the $i$-th row of $T$. Thus if the polynomial $g_j$ can be used to prove self-testing, then so can such a decomposition. This is especially important because a factorization of $Z$ will generally contain numbers of large bit-size (i.e., algebraic numbers of which the coefficients are rationals with large numerator and denominator), while the matrices $T$ can be much simpler. Intuitively, $T$ encodes exactly the information required for a strategy to be optimal.
}\fi

Our third contribution is to apply these techniques to the original CHSH mod 3 Bell inequality introduced by Buhrman and Massar in \cite{buhrman_causality_2005}. 
We give an exact certificate, which proves that the strategy of \cite{ji_multi-setting_2008} is optimal. 
Analytical self-testing proofs based on (concise) sum-of-squares certificates have been provided in  \cite{kaniewski2019maximal} and \cite{sarkar2021self}; however, those works treat different and more tractable inequalities than the original CHSH mod 3 Bell inequality. 
The latter work \cite{sarkar2021self} focuses on the SATWAP inequality proposed in \cite{salavrakos2017bell}. 
In the former work \cite{kaniewski2019maximal}, the CHSH mod 3 inequality is modified in such a way that a self-test statement can be proved. 
By contrast, our approach tackles the original CHSH mod 3 inequality itself, making it an ideal benchmark: although it does not appear to admit a simple sum-of-squares decomposition, it has numerically tight bounds and still allows the extraction of the optimal measurements.


%
{Closely related to self-testing is the problem of determining the optimal strategies: to prove that a Bell inequality yields a self-test, one must show that its maximal violation determines a unique optimal strategy.}
A well-known method to find such optimizers is by using an optimal solution to the dual semidefinite program: the moment matrix. 
Under a condition called flatness (also called the rank-loop condition \cite{navascues_convergent_2008}), this can be used to determine an optimal strategy \cite{burgdorf_optimization_2016}. 
Alternatively, one can follow the logic of self-testing proofs and use equations derived from an exact sum-of-squares certificate to recover an optimal strategy.
This can be done by using the equations directly as in \cite{cui_generalization_2020}, or, as noted in \cite{watts_noncommutative_2023}, by using a more general approach using Gr\"obner bases \cite{mora1994introduction}. 
Another contribution is to show that these two methods are directly related. 

The following theorem summarizes the main contributions above:
\begin{thmA}[Theorem \ref{thm:chsh_exact_violation} and Theorem \ref{thm:unique}]\label{thm:a}
The CHSH mod $3$ Bell function has maximal quantum value $\frac{1}{3} + \frac{2\cos(\pi/18)}{3\sqrt{3}}$. 
Moreover, up to unitary transformations and the natural symmetries of the Bell inequality, there is a unique corresponding irreducible strategy.
\end{thmA}
See the \ifx\compareversion\currentversion Preliminaries section in the Results \else Section~\ref{sec:rep_theory} \fi for a formal definition of irreducibility. 
%
We further use the positivity certificate underlying Theorem \ref{thm:a} to show that CHSH mod 3 yields, in a suitable sense, a robust self-test for the maximally entangled state of two qutrits.
More precisely, the symmetries of the defining polynomial give rise to multiple optimal strategies with non-equivalent measurements, but all of them use a maximally entangled state.
%
\begin{thmA}[Theorem \ref{thm:robust}]\label{thm:b}
The CHSH mod 3 Bell inequality robustly self-tests the maximally entangled state of qutrits. Specifically, if a strategy achieves a value within $\varepsilon$ of the maximal quantum value $\frac{1}{3} + \frac{2\cos(\pi/18)}{3\sqrt{3}}$, 
then, up to a local isometry, it is
$O(\sqrt{\varepsilon})$-close in norm to a direct sum of 
optimal, irreducible strategies.
In each optimal irreducible strategy, the underlying state is a maximally entangled pair of qutrits.
\end{thmA}

This paper is organized as follows. After some preliminaries, we recall the definition of the CHSH mod $d$ Bell inequality  \cite{buhrman_causality_2005}. 
We then specialize to the case $d=3$ and state the exact upper bound on the maximal quantum value $\beta_q$. 
After that, we consider two methods to extract optimal strategies from certificates, and show a new connection between the two methods. We also apply one of these methods to CHSH mod $3$, to determine all optimal strategies. 
We finish the Results section by establishing robust self-testing for CHSH mod $3$.
In the Methods section, we derive, using several reduction techniques, a tractable semidefinite program that yields an upper bound on the maximal value of the CHSH mod 3 Bell function.
We also apply a rounding scheme to obtain an exact rational solution for the reduced program. 
\ifx\compareversion\currentversion
\section*{Results}
\else 
\section{Results}
\fi

\ifx\compareversion\currentversion
\subsection*{Preliminaries}
\subsubsection*{Polynomial optimization}
\else 
\subsection{Preliminaries}
\subsubsection{Polynomial optimization}
\fi

\label{sec:prelim}
Let $X = (X_1, \dots, X_d)$ be a tuple of non-commuting variables. 
We denote by $\langle X \rangle$ the sets of words in $X$. A noncommuting polynomial $p \in \C\langle X \rangle$ is of the form
\[
p = \sum_{u \in \langle X \rangle} c_u u
\]
with finitely many nonzero coefficients $c_u$. The support of $p$, denoted by $\supp(p)$, is the set of words with nonzero coefficients. A word $u = \prod_{i=1}^n X_{j_i}$
is of degree $n$, and the degree of $p$ is the maximum degree of a word in the support of $p$. We denote by $\C\langle X\rangle_n$ the set of noncommutative polynomials of degree at most $n$. 

The algebra $\C\langle X \rangle$ is equipped with an involution $*$, which acts as complex conjugate on the coefficients and reverses words (i.e., $(\prod_{i=1}^n X_{j_i})^* = \prod_{i=1}^n X_{j_{n-i+1}}^*$). In this paper, we typically have $X_i^* = X_i^{-1}$. 

A two-sided ideal $\mathcal I$ of an algebra $\mathcal A$ generated by the elements $s_1, \dots, s_k \in \mathcal A$ is the set
\[
\langle s_i : i=1, \dots, k\rangle = \left\{\sum_{i,j} a_{ij} s_i b_{ij}  : a_{ij}, b_{ij} \in \mathcal A \right\},
\]
where the sum is finite. In this paper, the noncommutative variables are often partitioned into two tuples $X$ and $Y$, and part of the generators of the ideals we will use are then given by $X_iY_j - Y_jX_i$, so that the variables $X_i$ and $Y_j$ commute for all $i$ and $j$. 

A matrix $A \in \C^{N \times N}$ is positive semidefinite (resp.~positive definite), denoted by $A \succeq 0$ (resp.~$A \succ 0$) if it is Hermitian and all eigenvalues are nonnegative (resp.~positive). 
A Hermitian matrix has a spectral decomposition
\[
A = \sum_{i=1}^N \lambda_i \xi_i \xi_i^*,
\]
where $\xi^*$ is the conjugate transpose of $\xi\in \C^N$, and the square root of a positive semidefinite matrix is then given by
\[
\sqrt{A} = \sum_{i=1}^N \sqrt{\lambda_i} \xi_i \xi_i^*.
\]

Let $p \in \C\langle X, Y\rangle$ be a non-commutative polynomial in variables $X = (X_1, \dots, X_k)$ and $Y = (Y_1, \dots, Y_l)$, and consider (projection-valued) measurements $\{A_i\}_{i=1}^k$ and $\{B_j\}_{j=1}^l$ on separable Hilbert spaces $\mathcal H_A$ and $\mathcal H_B$ respectively, and a state $\psi \in \mathcal H_A \otimes \mathcal H_B$. The inequality 
\[ 
\beta(A, B, \psi) = \psi^* p(A\otimes I_B, I_A \otimes B) \psi \leq \beta_c,
\]
 where $\beta_c$ is the maximum value of $\beta(A, B, \psi)$ that can be obtained through a classical strategy (that is, $\psi = \psi_A \otimes \psi_B$ with $\psi_A \in \mathcal H_A$ and $\psi_B \in \mathcal H_B$), is called a Bell inequality. We denote the maximum value that can be obtained in quantum mechanics by $\beta_q$, and we call $(\{A_i\}_i, \{B_j\}_j, \psi)$ a strategy for the polynomial $p$, or simply a strategy when the polynomial is clear from the context. In general, we will consider commuting measurements $\{A_i\}$ and $\{B_j\}$ on the same Hilbert space $\mathcal H$.

Now, let $\mathcal{I}$ be the ideal of universal relations satisfied by all feasible measurement operators $A, B$. 
Suppose $g_1, \dots, g_N\in \C\langle X, Y\rangle$ are such that $p = \lambda -  \sum_j g_j^*g_j + q$ for some $\lambda \in \R$ and $q \in \mathcal I$, then
\begin{equation}\label{eq:poly_upperbound}
 \psi^* p(A, B) \psi = \lambda - \sum_j (g_j(A, B)\psi)^* g_j(A, B)\psi \leq \lambda
\end{equation}
for all strategies $(A, B, \psi)$. Thus $\lambda$ is an upper bound on $\beta_q$. This is the basis of non-commutative polynomial optimization. See \cite{burgdorf_optimization_2016} for a thorough introduction. 
Such $\lambda$, $q$ and $g_j$ can be found using semidefinite programming \cite{vandenberghe1996semidefinite}. 
Indeed, any sum-of-squares polynomial can be written as $v^*Zv$, where $Z$ is Hermitian positive semidefinite ($Z \succeq 0$), and $v$ is a so-called \emph{border vector} of which the entries form a basis of the non-commutative polynomials of degree at most the maximum degree of $g_j$.
The explicit semidefinite program can then be written as 
\begin{equation}\label{eq:general_sdp}
    \begin{aligned}
        & \text{inf} && \lambda \\
        & \text{s.t.} &&  \lambda - p
        = v^* Z v \mod \mathcal I,\\
        &&& Z \succeq 0.
    \end{aligned}    
\end{equation}
Solving such a semidefinite program gives a numerical solution, and one can generally find a rational sum-of-squares polynomial with a slightly worse $\lambda$ by relying on a so-called \emph{rounding and projection} algorithm. 
The initial rounding and projection algorithm has been applied for unconstrained polynomial optimization in \cite{parrilo}. 
Noncommutative extensions have been provided in \cite{cafuta2015rational,naceur2025certified}. 

By fixing the entries of the border vector $v$, this gives a finite semidefinite program. The idea of the NPA hierarchy is to increase the maximum degree step by step to get better bounds: the $n$-th level of the hierarchy sets $v = v_n$ to be the vector whose entries form a basis of the space of polynomials of degree at most $n$, and thus takes into account sum-of-squares polynomials of degree at most $2n$. 
Let $\mathcal H_A$ and $\mathcal H_B$ be Hilbert spaces. The partial trace $\Tr_A : \mathcal H_A \otimes \mathcal H_B \to \mathcal H_B$ is the unique linear map such that $\Tr_A(X \otimes Y) = \Tr(X)Y$ for all linear operators $X:\mathcal H_A \to \mathcal H_A$ and $Y:\mathcal H_B\to \mathcal H_B$.

\ifx\compareversion\currentversion
\subsubsection*{CHSH mod $d$}\label{sec:CHSH}
\else
\subsubsection{CHSH mod $d$}\label{sec:CHSH}
\fi
In this paper, we focus mainly on the CHSH mod $d$ Bell inequality originally introduced by Buhrman and Massar \cite{buhrman_causality_2005}. Fix a prime $d$ and define for all $i,j,k,l \in \{1,\dots,d\}$
\[
c_{i,j,k,l} = \frac{1}{d^2}\delta(i+j-kl \mod d),
\]
where $\delta(a) = 1$ if $a = 0$ and $0$ otherwise. Then the polynomial defining the Bell inequality is given by
\[
 p_d = \sum_{i,j,k,l=1}^d c_{i,j,k,l} A_i^k \otimes B_j^l.
\]
We wish to find Hilbert spaces $\mathcal H_A$ and $\mathcal H_B$, a state $\psi \in \mathcal H_A \otimes \mathcal H_B$ and projection-valued measurements 
$\{A_i^k:\mathcal H_A \to \mathcal H_A \mid k=1, \dots, d\}$ and $\{B_j^l:\mathcal H_B \to \mathcal H_B \mid l=1, \dots, d\}$ such that $\beta(A, B, \psi) = \psi^* \hat p_d(A, B) \psi$ is maximal. We denote this maximal $\beta(A,B,\psi)$ by $\beta_q$. Projection-valued measurements satisfy the conditions
\[
A_i^k A_j^k = \delta_{ij} A_i^k,\ \sum_{i} A_i^k = I, \ (A_i^k)^* = A_i^k,
\]
and likewise for the operators $B_j^l$.
The quantity $\beta(A, B, \psi)$ can be interpreted as the winning probability for a nonlocal game, where the players win if their answers $i, j \in \{1, \dots, d\}$ sum to the product of the questions $k, l \in \{1, \dots, d\}$ modulo $d$, and their strategy is measuring $\psi$ using the projection-valued measurements $A$ and $B$. The case $d = 2$ is the classical CHSH inequality, and in this paper we solve the case $d = 3$.

To formulate $p_d$ as a non-commutative polynomial, we use $A_i^k \otimes I$ and $I \otimes B_j^l$ as variables instead of $A_i^k$ and $B_j^l$, which effectively removes the tensor product and gives commutation relations $[A_i^k, B_j^l] = 0$.

Using the transformation
\[
X_k = \sum_{i=1}^d \omega^{-i} A_i^k, \quad Y_l = \sum_{j=1}^d \omega^{-j} B_j^l,
\]
where $\omega$ is a $d$-th root of unity, we can write the polynomial in terms of observables $X_j, Y_k$. From $\delta(x \mod d) = \frac{1}{d} \sum_{n=1}^d \omega^{nx}$ we obtain
\begin{align*}
    p_d &= \frac{1}{d^3}\sum_{i,j,k,l,n=1}^d \omega^{n(-i-j+kl)} A_i^kB_j^l \\
    &= \frac{1}{d^3}\sum_{k,l,n=1}^d \omega^{kln}(\sum_{i=1}^d\omega^{-i}A_i^k)^n(\sum_{j=1}\omega^{-j}B_j^l)^n \\
    &
    = \frac{1}{d^3}\sum_{k,l,n=1}^d \omega^{kln} X_k^nY_l^n,
\end{align*}
where in the second equality we used that $A_i^k$ and $B_j^l$ are projections.
Since $A_i^k$ and $B_j^l$ form projection-valued measurements, 
$X_k$ and $Y_l$ are $d$-th roots of the identity operator, and $X_k^* = X_k^{-1}$, $Y_l^*=Y_l^{-1}$. Since $A$ and $B$ commute, so do $X$ and $Y$. The variables $X_j$ and $Y_k$ generate a group.

We denote by $\mathcal I$ the ideal generated by the relations the variables $X$ and $Y$ satisfy, i.e., 
\begin{equation}
    \label{eq:ideal}
\mathcal I = \langle X_j Y_k-Y_kX_j,\, X_j^d - I,\, Y_j^d-I \mid j, k \in \{1, \dots, d\} \rangle.
\end{equation}
For reference, the non-commutative polynomial optimization problem we consider in the remainder of this paper is 
\begin{equation}\label{eq:ncpop}
\begin{aligned}
    \beta_q &= && \text{sup} && \psi^* p_d(X, Y) \psi \\
    &&&\text{subject to} && \mathcal H  && \text{Hilbert space},\\
    &&&&& X_i, Y_i: \mathcal H \to \mathcal H, && \text{with } q(X, Y) = 0 \quad  \forall q \in \mathcal I \\
    &&&&& \psi \in \mathcal H.
\end{aligned}
\end{equation}
Typically, we take $d = 3$, which will be clear from the context.

\ifx\compareversion\currentversion
\subsubsection*{Representation theory}
\else
\subsubsection{Representation theory}
\label{sec:rep_theory}
\fi
We denote the identity element of a group by $e$. The direct product of two groups $G_1, G_2$ is given by the group $G_1 \times G_2 = \{ (\zeta_1, \zeta_2) : \zeta_1 \in G_1, \zeta_2 \in G_2\}$ with the product $(\zeta_1, \zeta_2) \cdot (\zeta_3, \zeta_4) = (\zeta_1\zeta_3, \zeta_2\zeta_4)$. Given a homomorphism $\phi:G_2 \to \mathrm{Aut}(G_1)$, the semidirect product $G_1 \rtimes G_2$ uses the same set of elements, with the product $(\zeta_1, \zeta_2) \cdot (\zeta_3, \zeta_4) = (\zeta_1 \phi(\zeta_2)(\zeta_3), \zeta_2 \zeta_4)$. That is, instead of commuting variables $(\zeta_1, e)$ and $(e, \zeta_2)$, the variables satisfy the relation $(e, \zeta_2)\cdot (\zeta_1, e) = (\phi(\zeta_2)(\zeta_1), e)\cdot (e, \zeta_2)$. The group $G_1 \simeq G_1 \times \{e\}$ is a normal subgroup of $G_1 \rtimes G_2$: for every $\zeta_1 \in G_1, \zeta_2 \in G_2$ we have $(e, \zeta_2) \cdot (\zeta_1, e) \cdot (e, \zeta_2^{-1}) = (\phi(\zeta_2)(\zeta_1), e) \in G_1 \times \{e\}$. 

A representation $\pi$ of a group $G$ on a vector space $V_\pi$ is a group homomorphism $\pi:G \to \GL{V_\pi}$. We refer to both $\pi$ and the associated vector space $V_\pi$ as a representation.  
The dimension $d_\pi$ of the representation $\pi$ is the dimension of $V_\pi$. 
A representation is irreducible if the only subspaces $W \subseteq V_\pi$ such that $\pi(\gamma)W \subseteq W$ for every $\gamma \in G$ are $V_\pi$ and $\{0\}$. Two representations $(\pi, V_\pi), (\pi', V_{\pi'})$ are equivalent if there is an invertible map $T:V_\pi \to V_{\pi'}$ with $T\pi(\gamma) = \pi'(\gamma)T$ for all $\gamma \in G$ (i.e., $T$ is equivariant). For more background on representation theory, see for example \cite{serre_linear_1996, fulton_representation_1991}.

\ifx\compareversion\currentversion
\subsection*{Symmetries of CHSH mod $d$}\label{sec:p_symmetries}
\else
\subsection{Symmetries of CHSH mod $d$}\label{sec:p_symmetries}
\fi
The polynomial $p_d$ admits many symmetries. Such symmetries can be exploited to drastically reduce the size of the semidefinite programs used to compute bounds. The symmetries the polynomial $p_d$ has are generated by the following actions:
\begin{itemize}
    \item Interchanging $X_i$ with $Y_i$ for all $i$ simultaneously:
    \begin{equation}\label{eq:sym1}
        (X, Y) \mapsto (Y_1, \dots, Y_d, X_1, \dots, X_d)
    \end{equation}
    \item Negating all indices modulo $d$: 
    \begin{equation}\label{eq:sym2}
        (X, Y) \mapsto (X_{d-1}, \dots, X_1, X_d, Y_{d-1}, \dots, Y_1, Y_d)
    \end{equation}
    \item Increasing the index of either $X$ or $Y$ and multiplying the other by a power of $\omega$ depending on the index: 
    \begin{equation}\label{eq:sym3}
    \begin{aligned}
        (X, Y) &\mapsto (X_2, \dots, X_d, X_1, \omega^1 Y_1, \dots, \omega^d Y_d),  \\
        (X, Y) &\mapsto (\omega^1 X_1, \dots, \omega^d X_d, Y_2, \dots, Y_d, Y_1) \\
    \end{aligned}
    \end{equation}
    \item Inverting the matrices, and negating the indices of either the $X$ matrices or the $Y$ matrices modulo $d$: 
    \begin{equation}\label{eq:sym4}
        \begin{aligned}
            (X, Y) &\mapsto (X_1^{d-1}, \dots, X_d^{d-1}, Y_{d-1}^{d-1}, \dots, Y_{1}^{d-1}, Y_d^{d-1}), \\
            (X, Y) &\mapsto (X_{d-1}^{d-1}, \dots, X_1^{d-1}, X_d^{d-1}, Y_1^{d-1}, \dots, Y_d^{d-1})
        \end{aligned}
    \end{equation}
\end{itemize}
Except for the last symmetry, these maps do not influence the total degree of a word in the variables $X$ and $Y$. The group generated by \eqref{eq:sym1}-\eqref{eq:sym3} is $\Gamma = (C_d \times C_d) \rtimes (C_2 \times C_2)$, where $C_j$ is the cyclic group with $j$ elements.

\ifx\compareversion\currentversion
\subsection*{Upper bound on $\beta_q$ for CHSH mod $3$}
\else
\subsection{Upper bound on $\beta_q$ for CHSH mod $3$}
\fi
For our choice of the border vector $v$ and the formulation of our final semidefinite program, see
\ifx\compareversion\currentversion
the Methods section.
\else
Section~\ref{sec:methods}.
\fi
This results in some vectors $v_{\pi,j}$ whose entries are noncommutative polynomials such that the constraint of the semidefinite program reads
\begin{equation}\label{eq:final_constraint}
\lambda - p_3 = \sum_{\pi \in \hat \Gamma} \sum_{j=1}^{d_\pi} v_{\pi,j}^* \begin{pmatrix}
    I & \sqrt{3/4}\mathrm{i}I
\end{pmatrix}Z^{\pi}\begin{pmatrix}
    I \\ -\sqrt{3/4}\mathrm{i}I
\end{pmatrix}v_{\pi, j} \mod \mathcal I,
\end{equation}
where the matrices $Z^{\pi}$ are real positive semidefinite matrix variables, $\hat \Gamma$ are the irreducible representations of the group $\Gamma$, and $\mathrm{i}$ is the imaginary unit.
\begin{theorem}
\label{thm:chsh_exact_violation}
The maximal value of the CHSH mod $3$ Bell function is at most $\frac{1}{3} + \frac{2\cos(\pi/18)}{3\sqrt{3}}$.
\end{theorem}
\begin{proof}
Solving the semidefinite program \eqref{eq:general_sdp} where the constraint is specialized to \eqref{eq:final_constraint},
and rounding the solution using the rounding procedure of \cite{cohn_optimality_2024} gives a solution over the number field $F$ with generator $z \approx 1.5320889$ satisfying $1 - 3z + z^3 = 0$. The matrices in the exact solution returned by the rounding procedure are of the form
\[
    Z^\pi = T_\pi \hat Z^\pi T_\pi^{\sf T}
\]
with $\hat Z^\pi \succ 0$ (the matrices $T_\pi$ are in general not square), where the entries of $\hat Z^\pi$ and $T_\pi$ are elements in $F$. 
The exact solution is feasible with objective function value
\[
\lambda = \frac{1}{9}(1 + 2z + z^2) = \frac{1}{3} + \frac{2\cos(\pi/18)}{3\sqrt{3}}, 
\]
which shows that this is an upper bound on the maximal value of CHSH mod $3$.

To verify that the solution is indeed feasible, we check that the affine constraints \eqref{eq:SDP_constraints_final} hold, and that the matrices $\hat Z^{\pi}$ are positive definite in interval arithmetic. The solution and the code to verify the feasibility of the solution are available at \cite{code_chshmod3}. 
\end{proof}
Note that this proves that the construction of Ji et al. in \cite{ji_multi-setting_2008} is optimal.

\ifx\compareversion\currentversion
\subsection*{Optimizer extraction}\label{sec:extraction}
\else
\subsection{Optimizer extraction}\label{sec:extraction}
\fi
We consider two  methods to extract optimizers from a sharp semidefinite programming bound. First we use the exact sum-of-squares certificate to find an ideal $\mathcal J$ such that $q(X, Y, \psi) = 0$ for any $q \in \mathcal J$ and any strategy $(X, Y, \psi)$ maximizing $\beta(X, Y, \psi)$. If the group generated by any optimal operators $X, Y$ is finite, all possible optimizers can be extracted, up to unitary transformations. The extraction method is based on \cite[Section 6.3]{watts_noncommutative_2023}. 

After that we consider a well-known technique that requires flatness of the dual certificate, the moment matrix. See for example \cite{burgdorf_optimization_2016}. The two methods are closely related to each other. We show that if the moment matrix is flat, then under mild conditions the two extraction methods lead to the same optimizers.
For the second level of the hierarchy introduced in \ifx\compareversion\currentversion the Methods section \else Section~\ref{sec:methods} \fi for CHSH mod $3$, this method cannot be used because the resulting moment matrix is not flat.

In the following two sections, we consider a slightly more general polynomial optimization problem than a Bell scenario with two parties. We take a polynomial $p \in \C\langle X \rangle$, and consider an ideal $\mathcal I$ such that the variables $X_i$ generate a group modulo the ideal. Furthermore, we assume that $X_i$ is unitary for all $i$, i.e., the involution is defined by $X_i^* = X_i^{-1}$.

Recall that the constraint is of the form $\lambda - p = v^{*} Z v \mod \mathcal I$, with $Z$ Hermitian positive semidefinite and $v$ a basis of a vector space of polynomials, such that $p+q \in \Span\{a^*b : a, b \in v\}$ for some $q \in \mathcal I$. If $v$ is a vector of words, the dual semidefinite program to \eqref{eq:general_sdp} has a simple form and can be written as
\begin{equation}\label{eq:moment_sdp}
\begin{aligned}
& \max && \langle G_p, M \rangle, \\
& \text{subject to} && M_{1,1} = 1, \\
&&& M_{a, b} = M_{x,y} & \text{if }a^*b = x^*y \mod \mathcal I \\
&&& M \succeq 0,
\end{aligned}
\end{equation}
where $G_p$ is a matrix such that $p = v^*G_pv \mod \mathcal I$. The matrix $M$ is referred to as the moment matrix and is indexed by $a, b \in v$. 

\ifx\compareversion\currentversion
\subsubsection*{Extraction through SOS certificates}\label{sec:grobner_optimizer}
\else
\subsubsection{Extraction through SOS certificates}\label{sec:grobner_optimizer}
\fi
Let $(X, \psi)$ be a strategy that maximizes $\psi^* p(X)\psi$, and suppose $(\beta_q, Z)$ is an optimal SOS certificate. Then in particular
\[
0= \beta_q - \psi^* p(X)\psi  = \psi^* v^*(X) Z v(X)\psi.
\]
Now suppose $Z = T\hat Z T^*$, with $\hat Z \succ 0$. Then for any optimal strategy $(X, \psi)$, we have that
\[
0 = \psi^* v^*(X) Z v(X) \psi = \psi^* v^*(X) T\hat Z T^* v(X) \psi = \|\sqrt{\hat Z} T^*v(X)\psi\|^2
\]
where $\sqrt{\hat Z}$ is the square root of $\hat Z$. Since $\sqrt{\hat Z}$ is an invertible matrix, we have for every column $T_i$ of $T$ that
\[
T_{i}^* v(X) \psi = 0.
\]
That is, any optimal strategy $(X, \psi)$ satisfies $q(X, \psi) = 0$ for any $q$ in the two-sided ideal $\mathcal J \subseteq \C\langle X, \psi \rangle$ generated by  $\{T_i^*v(X)\psi\}_i$ and generators of $\mathcal I$. 

Now define $H_{\mathcal J} = \C\langle X\rangle \psi /\mathcal J$, and consider the map $\rho: \{X_i\}_i \to \mathcal{L}(H_{\mathcal J})$ defined by $\rho(X_i)u = X_iu$, and extend this to $\C\langle X \rangle/\mathcal I$. Here $\mathcal{L}(H_{\mathcal J})$ denotes the space of linear operators on $H_{\mathcal J}$. 
Then the matrices $\rho(X_i)$ satisfy $q(\rho(X))u=\rho(q(X))u = q(X)u = 0$ for all $q \in \mathcal I$ and $u \in H_{\mathcal J}$, so in particular the matrices $\rho(X_i)$ generate a group $G$. We assume $G$ to be finite; note that this in particular implies that $H_{\mathcal J}$ is finite dimensional. Then $\rho$ is a representation of $G$ when restricted to words. 

Take an inner product $\langle \cdot, \cdot \rangle$ on $H_{\mathcal J}$ such that $\rho$ is a unitary representation. For example, given any inner product $( \cdot, \cdot)$, take the inner product $\langle u, w\rangle = \frac{1}{|G|}\sum_{\zeta \in G}( \rho(\zeta)u, \rho(\zeta)w )$. Then $H_{\mathcal J}$ is a Hilbert space. 
Moreover, if we extend $\rho$ by linearity to $\C\langle X \rangle$, we have
\[
\langle \psi, \rho(p) \psi\rangle = \langle\psi, \rho(\beta_q I - v^*(X)Zv(X))\psi\rangle =\beta_q \langle \psi, \psi\rangle
\]
because $q(X,\psi) = 0$ for any $q \in \mathcal J$ and $v^*(X)Zv(X)\psi \in \mathcal J$. Thus, $(\rho(X), \psi/\|\psi\|)$ is an optimal strategy.

\begin{remark}
If $G$ is infinite but compact, we can still average the inner product over the group using its Haar measure. Therefore, this will still give a unitary representation and an optimal (but possibly infinite-dimensional) strategy. 
\end{remark}
\begin{remark}
    If the variables $X$ do not generate a group and $H_{\mathcal J}$ is finite, the same method can be used to find matrices $\rho(X)$ and a state $\psi$ that satisfy almost all requirements by choosing a basis of $H_{\mathcal J}$. Since the ideal does not enforce conditions on the adjoint of the variables (i.e., $X$ must be Hermitian, or unitary), such conditions are typically not directly satisfied by $\rho(X)$ in a chosen basis. In the next section, an inner product for which the adjoint conditions are satisfied comes from the moment matrix, i.e., the solution to the dual semidefinite program. On the sum-of-squares side, however, it is not directly clear how to define a suitable inner product.
\end{remark}

Using representation theory, we can block-diagonalize $\rho$. 
Suppose that $\{(\rho_{k}, V_{k})\}_k$ is a complete set of (unitary) irreducible representations of $G$. Then $\rho$ can be block-diagonalized as 
\[
P\rho P^{-1} = \bigoplus_{k} \bigoplus_{i=1}^{m_{k}} \rho_{k}^i,
\]
where the irreducible representations $\rho_{k}^i$ are equivalent to $\rho_{k}$ for each $i$, and $m_{k}$ is the multiplicity of $\rho_{k}$ in $\rho$. We denote the subspace of $H_{\mathcal J}$ on which $\rho_{k}^i$ acts by $H_{k}^i$.  Since both $\rho_{k}^i$ and $\rho$ are unitary, the basis transformation matrix $P$ is unitary. Furthermore, each subrepresentation $\rho_{k}^i$ of $\rho$ gives an optimal strategy $(\rho_{k}^i(X), \psi_{k,i}/\|\psi_{k,i}\|)$, where $P\psi = \bigoplus_{k, i}\psi_{k}^i$ is a decomposition with $\psi_{k}^i \in H_{k}^i$.

We call a strategy $(X, \psi)$ with $\psi \in H$ irreducible if there is no subspace $V$ of $H$ such that $X_iV \subseteq V$ for all $i$. In particular, direct sums of optimal strategies are reducible.
\begin{lemma}
    If $(X, \psi)$ is optimal and irreducible, then there is some state $\hat \psi$ such that $(X, \psi)$ is unitarily equivalent to $(\rho_{k},\hat\psi)$ for some $k$.
\end{lemma}
\begin{proof}
Define the representation $\pi:G \to H_\psi$, where $H_\psi$ is the Hilbert space $\psi$ lives in, with $\pi(X) = X$. Note that a strategy is irreducible if and only if this representation is irreducible. Hence it is equivalent to $\rho_{k}$ for some $k$, and since both representations are unitary, they are unitarily equivalent. That is, there is some unitary bijection $T: H_\psi \to H_{k}$ such that 
\[
\rho_k = T \pi T^{-1}.
\]
Set $\hat \psi = T\psi$. This gives a strategy $(\rho_k, \hat \psi)$ unitarily equivalent to $(X, \psi)$.
\end{proof}
Note that this only says that all optimal irreducible strategies can be found among the irreducible representations of the group $G$. However, in principle the multiplicity $m_k$ could be $0$ for some optimal irreducible representation $\rho_k$. The next theorem shows that this is not the case.
\begin{theorem}\label{thm:all_opt_strats}
    The strategies $(\rho_k^i(X), \psi_{k,i})$ with $\psi_{k,i} \neq 0$ are all optimal irreducible strategies.
\end{theorem}
\begin{proof}
    Suppose $(\pi, \hat \psi)$ is an irreducible strategy but not equivalent to any of the strategies $(\rho_k^i, \psi_k^i)$. Then the projection
    \[
    p_{11}^\pi = \frac{d_\pi}{|G|}\sum_{\zeta \in G} \pi(\zeta^{-1})_{11} \rho(\zeta)
    \]
    is the zero map from $H_{\mathcal J}$ to $H_{\mathcal J}$ by \cite[Proposition~8]{serre_linear_1996}. Let $U\psi \subseteq H_{\mathcal J}$ be a basis. Then, for any element $u \in U$, we have
    \[
    0 = p_{11}^\pi u\psi = q(X, \psi),
    \]
    for some $q \in \mathcal J$.
     Now consider the evaluation of $q$ on $(\pi(X), \hat \psi)$. This gives
    \[
    \frac{d_\pi}{|G|}\sum_{\zeta \in G} \pi(\zeta^{-1})_{11} \rho(\zeta) u(\pi(X)) \hat \psi = \frac{d_\pi}{|G|}\sum_{\zeta \in G} \pi(\zeta^{-1})_{11} \pi(\zeta) u(\pi(X)) \hat \psi
    \]
    which is the projection of $H_\pi$ onto itself, and is nonzero if the first entry of $u(\pi(X))\hat\psi$ is nonzero. In particular, $(\pi(X), \hat \psi)$ does not satisfy $q(\pi(X), \hat \psi)$ for all $q \in \mathcal J$, and is therefore not optimal by the sum-of-squares certificate. 
%
\end{proof}

We now apply this to CHSH mod $3$. 
\begin{theorem}\label{thm:unique}
The polynomial $p_3$ has a unique irreducible strategy $(X, Y, \psi)$ that optimizes problem \eqref{eq:ncpop}, up to unitary transformations and symmetries of the polynomial $p_3$ generated by \eqref{eq:sym1}-\eqref{eq:sym4}.
\end{theorem}
\begin{proof}
Let $(\beta_q, \bigoplus_\pi T_\pi \hat Z_\pi T_\pi^{\sf T})$ be the exact sum-of-squares certificate used in the proof of Theorem~\ref{thm:chsh_exact_violation}, and let $\{v_{\pi,j}\}$ be the vectors containing the symmetry adapted basis such that
\[
\beta_q - p_3 = \sum_{\pi} \sum_{j=1}^{d_\pi} v_{\pi,j}^* \begin{pmatrix} I  & \sqrt{3/4}\mathrm{i}I\end{pmatrix} T_\pi \hat Z_\pi T_\pi^{\sf T} \begin{pmatrix} I  \\ -\sqrt{3/4}\mathrm{i}I\end{pmatrix} v_{\pi, j} \mod \mathcal I.
\]
Let $\mathcal J \subseteq \C\langle X, Y, \psi\rangle$ be the two-sided ideal generated by the standard relations on $X_i, Y_j$ (commutation, idempotency), together with the polynomials 
\[
T_{\pi, i}^{\sf T} \begin{pmatrix}
    I \\ -\sqrt{3/4}\mathrm{i}I
\end{pmatrix} v_{\pi, j}(X, Y) \psi 
\]
for every column $T_{\pi, i}$ of $T_\pi$. We use Nemo.jl and Hecke.jl \cite{fieker_nemohecke_2017} to compute a non-commutative Gr\"obner basis for $\mathcal J$, and define the representation $\rho$ as before. The matrices $\rho(X_i)$ form the group
\begin{equation}
    \label{eq:group}
    G = \langle X_1, X_2, X_3 : X_i^3 = I, X_i X_j X_k = X_kX_iX_j \, \text{for all}\, i\neq j\neq k\neq i\rangle,
\end{equation}
where it can be checked that $(f_1-f_2)\psi \in \mathcal J$ for each equality $f_1=f_2$ in the definition of the group by reducing it with respect to the Gr\"obner basis. 
The group $G$ is isomorphic to the group $C_3 \times ((C_3 \times C_3) \rtimes C_3)$, which is the group 81.12 from the SmallGroups library \cite{besche_smallgrp_2024} in GAP \cite{noauthor_gap_2025}. The same holds for the group generated by $\rho(Y_i)$, so the group generated by all operators is given by $G \times G$. Note that $(C_3 \times C_3) \rtimes C_3$ is the Heisenberg-Weyl group on 3 elements. 

We obtain the irreducible representations of $G$ from GAP, and the irreducible representations of $G \times G$ are tensor products of pairs of irreducible representations of $G$ by \cite[Theorem~10]{serre_linear_1996}. Trying all irreducible representations of $G \times G$ shows that there are $4$ irreducible representations that give an optimal strategy. 

Alternatively, we can directly block-diagonalize $\rho$, which gives all optimal strategies by Theorem~\ref{thm:all_opt_strats}. This gives $4$ tuples of $9 \times 9$ matrices, where each matrix can be further decomposed as a tensor product between two $3 \times 3$ matrices, such that $X_i = \hat X_i \otimes I_3$ and $Y_j = I_3 \otimes \hat Y_j$. 

Using the Jordan normal form, we apply transformations to simplify the matrices. One of the tuples then gives the matrices
\begin{equation*}
\hat X_1 = Z^2X^2,\; \hat X_2 = X,\; \hat X_3 = Z,\;\; \hat Y_1 = X,\; \hat Y_2 = Z^2X^2,\; \hat Y_3 = Z, 
\end{equation*}
where $X$ and $Z$ are matrices acting on the vector space spanned by $|j\rangle$ for $j=0, \dots, 2$, with $X|j\rangle = |j+1 \mod 3\rangle$ and $Z|j\rangle = \omega^j|j\rangle$, where $\omega=\exp(\frac{2\pi \mathrm{i}}{3})$. The other tuples are (unitary transformations of) the result of applying the transformations
\[
(\hat X_i, \hat Y_j) \mapsto (\hat Y_i, \hat X_j)
\]
and/or
\[
(\hat X_i, \hat Y_j) \mapsto (\hat X_i^{-1}, \hat Y_{(-i \mod 3)}^{-1})
\]
on this tuple. The states corresponding to the tuples are all equal to the state
\[
\psi = c(1, z-1, \omega^{-1}(-z^2+2), z-1, -z^2+2, \omega^{-1}, \omega^{-1}(-z^2+2), \omega^{-1}, \omega(z-1))
\]
where $z \approx 1.5320889$ satisfies $1-3z+z^3 = 0$, and $c = \sqrt{-9z + 18}$ is a normalizing constant. The Julia code to verify that the equalities defining the groups generated by $\rho(X_i)$ and $\rho(Y_i)$ are as above, to find and simplify the tuples of matrices, and to verify the equivalences, is available at \cite{code_chshmod3}.
\end{proof}
A state $\psi$ is maximally entangled if the reduced states $\Tr_A(\psi\psi^*)$ and $\Tr_B(\psi\psi^*)$ are maximally mixed, i.e., equal to $1/\dim (B) I_B$ and $1/\dim(A) I_A$, respectively. It can easily be checked that this is the case for the state given in the proof of Theorem~\ref{thm:unique}.
\ifx\compareversion\currentversion
\subsubsection*{Flatness}\label{sec:flatness}
\else
\subsubsection{Flatness}\label{sec:flatness}
\fi
In this section, we assume that the entries of the border vector $v_n$ form a basis of the polynomials in $\C\langle X \rangle_n/\mathcal I$. Without loss of generality we may order the entries of the vectors such that $v_{n-1}$ is the first part of $v_n$. Let $M_n$ be the corresponding moment matrix. 

As will be explained in \ifx\compareversion\currentversion the Methods section, \else Section~\ref{sec:SOS_cond_exp} \fi if $\C\langle X \rangle/\mathcal I$ is a group algebra, one can use the support of the polynomial $p$ together with $1$ as $v_1$ as border vector.
In that case, the variables that can be extracted using flatness are the elements in the support of $p$. 


Let $\delta$ be such that the generators of $\mathcal I$ are of degree at most $2\delta$. 
A moment matrix $M_n$ is called $\delta$-\emph{flat} if the rank of the restriction $M_{n-\delta}$ corresponding to $v_{n-\delta}$ is equal to the rank of $M_n$. Flatness of an optimal solution implies optimality (i.e., increasing the level of the hierarchy will not improve the bound anymore and $\langle G_p, M_n\rangle = \beta_q$) \cite{navascues_convergent_2008}, and can be used to extract a minimizer.

Let $M_n = R_n^{*}R_n$ be a Gram decomposition of $M_n$. Then since $M_n$ is flat, the Gram vectors corresponding to words of degree $n$ can be expressed in terms of the Gram vectors of the words up to degree $n-\delta$. Let $\{w_a\}_{a \in U}$ be a basis of the column space of $R_{n-\delta}$, where $w_a$ is the column corresponding to a word $a$ and $U \subseteq \C\langle X \rangle_{n-\delta}/\mathcal I$. Let $V = \Span\{w_a\}_{a \in U}$. Define the function $\rho:\{X_i\}_i \to \mathcal L(V)$ by $\rho(X_i)w_a = w_{X_ia}$. Since $H_n$ is flat, $w_{X_iu}$ is a linear combination of the vectors $\{w_u\}_{u \in U}$, so $\rho(X_i)$ indeed maps vectors from $V$ to $V$. 

The matrix $M_n$ defines a linear functional $L:\C\langle X \rangle_{2n} /\mathcal I\to \C$ by $L(p^*q) = M_{p,q}$, and the inner product $\langle w_p, w_q\rangle = L(p^*q)$ makes $V$ a Hilbert space. The matrix $\rho(X_i)$ is unitary with respect to the inner product, because $\langle \rho(X_i)w_a, \rho(X_i)w_b\rangle_M = L((X_ia)^*X_ib) = L(a^*b) = \langle w_a, w_b\rangle_M$ for all $a,b \in U$ due to the constraints on $M$ in \eqref{eq:moment_sdp}. In particular, this means that $\rho(X_i)^* = \rho(X_i^*)$. Let $q \in \mathcal I$ be of degree at most $2\delta$. Then 
\begin{align*}
\langle w_a, q(\rho(X))w_b \rangle &= \sum_{j} c_j \langle w_a, \prod_{i=1}^{|j|} \rho(X_{j_i}) w_b\rangle \\
&= \sum_j \langle \prod_{i=1}^{\max\{0,|j|-\delta\}}\rho(X_{j_{|j|-\delta-i+1}})^*w_a, \prod_{i=\max\{1, |j|-\delta+1\}}^{|j|}\rho(X_{j_i}) w_b\rangle\\
&= \sum_j c_j L(a^*\prod_{i=1}^{|j|} X_{j_i}b) \\
&= L(a^*q(X)b) = 0.
\end{align*}
So the matrices $\rho(X_i)$ satisfy the same relations as $X_i$. In particular, they generate a group $G$, and as in the previous section we assume that $G$ is finite. This gives a representation of $G$ on $V$.

Furthermore, $\langle w_1, w_1\rangle_M = (M_n)_{1,1} = 1$, and 
\[
\langle w_1, p(\rho(X))w_1\rangle_M = \langle w_1, \sum_{a\in U}p_a a(\rho(X))w_1\rangle_M = \sum_{a\in U} p_a L(a) = \langle G_p, M_n \rangle,
\] 
where we write $a(X)$ for the evaluation of the word $a$ at the matrices $X$. Note that the inner product between $G_p$ and $M_n$ is the trace inner product between two matrices. This shows that $(\rho(X), w_1)$ is a feasible solution with $\beta(\rho(X), w_1) = \langle G_p, M_n\rangle = \beta_q$. 


In the following theorem, we use that the moment matrix (used in this section) and the sum-of-squares certificate (used in the previous section) come from dual semidefinite programs to show that the methods lead to the same construction.
\begin{theorem}\label{thm:equiv_flat_groebner}
    Let $(\lambda, Z; M) \in \R \times \C^{N \times N} \times \C^{N \times N}$ be a primal-dual optimal solution with $\mathrm{rank} Z + \mathrm{rank} M = N$ and $\lambda = \langle G_p, M\rangle$. If $M$ is $\delta$-flat, then $H_{\mathcal J}$ is finite dimensional, and the representations defined in the previous sections are equivalent.
\end{theorem} 
We provide the proof of this theorem in \ifx\compareversion\currentversion the Supplementary material, \else Appendix~\ref{sec:pf_thm_equiv}, \fi and give here a sketch of the proof.
\begin{proof}[Sketch of the proof.]
    From semidefinite programming duality, we obtain $\langle Z, M \rangle = 0$. Together with $\rank Z + \rank M = N$, this allows us to equate the nullspace of $M$ to the column space of $Z$. Since the nullspace of $M$ gives relations satisfied by the representation defined using flatness, and the column space of $Z$ defines the ideal used to define the representation in the previous section, this gives the desired connection between the two representations.
\end{proof}

\ifx\compareversion\currentversion
\subsection*{Robust self-testing with CHSH mod $3$}\label{sec:selftesting}
\else
\subsection{Robust self-testing with CHSH mod $3$}\label{sec:selftesting}
\fi
Theorem~\ref{thm:unique} gives us the only possible shapes an optimal strategy can have: the state is a direct sum of scaled maximally entangled states, possibly extended with an auxiliary state through a tensor product, and the observables are direct sums of the corresponding irreducible representations, possibly extended with the identity for the auxiliary state. In this section, we make this statement robust.   

Let $G$ be a finite group and $\varepsilon \geq 0$. For the majority of this section, we take $G$ to be the group defined in \eqref{eq:group}, but the following definition and Theorem~\ref{thm:GH} hold for general groups $G$. Let $\mathcal H_A$ and $\mathcal H_B$ Hilbert spaces of dimensions $n_A$ and $n_B$, with $\psi \in \mathcal H_A \otimes \mathcal H_B$, and $R = \Tr_B(\psi\psi^*)$ the reduced density matrix for system $A$. We denote the group of unitary matrices of size $n\times n$ by $U_n(\C)$.
A function $f:G \to U_{n_A}(\C)$ is an $(\varepsilon, \psi)$-representation for $G$ if 
\[
\frac{1}{|G|^2}\sum_{x,y \in G}\|f(x)f(y)^* - f(xy^{-1})\|_R^2 \leq \varepsilon,
\]
where $\|A\|_R^2 = \Tr(AA^*R)$. 

Gowers and Hatami showed that an $(\varepsilon, \psi)$-representation is $\varepsilon$-close to an actual representation:
\begin{theorem}[Gowers-Hatami \cite{gowers_inverse_2017}]\label{thm:GH}
    Let $G$ be a finite group and suppose $f:G \to U_{n_A}(\C)$ is an $(\varepsilon, \psi)$-representation for $G$. Then there is some $n_A' \geq n_A$, a representation $\tau:G \to U_{n_A'}(\C)$ of $G$ and an isometry $U:\C^{n_A} \to \C^{n_A'}$ such that 
    \[
    \frac{1}{|G|} \sum_{x \in G}\|f(x) - U^* \tau(x) U\|_R^2 \leq \varepsilon.
    \]
\end{theorem}
From the proof by Vidick \cite{vidick_pauli_2017}, the representation $\tau$ can be decomposed as $\bigoplus_{\pi} I_{n_A} \otimes I_{d_\pi} \otimes \pi$, where the direct sum runs over all irreducible representations of $G$. Of course, it is possible to replace $n_A$ by $n_B$, and to take $R = \Tr_A(\psi\psi^*)$.

Recall that the group $G$ in \eqref{eq:group} is isomorphic to $H = C_3 \times ((C_3 \times C_3) \rtimes C_3)$, the group 81.12 from the SmallGroups library from GAP \cite{besche_smallgrp_2024}. The isomorphism $\phi: H \to G$ is defined by 
\begin{equation}\label{eq:Xi_to_fi}
\begin{aligned}
    \phi(\gamma_1) &= X_1^2X_2^2,\\
    \phi(\gamma_2) &= X_1^2X_3(X_3X_2^{-1}X_3^{-1}X_2)^2X_3,\\
    \phi(\gamma_3) &= X_1(X_2X_3X_2^{-1}X_3^{-1})^4X_2X_3,\\
    \phi(\gamma_4) &= (X_2^{-1}X_3^{-1}X_2X_3)^4.
\end{aligned}
\end{equation}
The generators $\gamma_1, \dots, \gamma_4$ of $H$ satisfy $\gamma_i^3 = I$ for all $i$, $[\gamma_i, \gamma_j] = 0$ if $j \in \{3,4\}$, and $\gamma_2\gamma_1 = \gamma_4\gamma_1\gamma_2$. The elements of $H$ are of the form $\prod_{i=1}^4 \gamma_i^{j_i}$ with $j \in \{0, 1, 2\}^4$. We usually write $\gamma_i(X_1, X_2, X_3)$ or $\gamma_i(X)$ for $\phi(\gamma_i)$ evaluated on $X=(X_1, X_2, X_3)$ (where $(X_1, X_2, X_3)$ 
are matrices that do not necessarily satisfy the relations defining the group $G$), or $\gamma_i$ if it is clear from the context that we mean the evaluated isomorphism and not the generators of the group $H$.

\begin{lemma}\label{lem:eps_psi_rep}
    Suppose $(X \otimes I, I \otimes Y, \psi)$ is a feasible strategy 
    with $\beta_q - \psi^* p_3(X \otimes I, I \otimes Y)\psi \leq \varepsilon$. Then there is some $\varepsilon' = O(\varepsilon)$ such that $f_A:G \to U_{n_A}(\C)$ and $f_B:G \to U_{n_B}(\C)$ defined by
    \[
    f_A(\phi(\prod_i \gamma_i^{j_i})) = \prod_i \gamma_i(X_1, X_2, X_3)^{j_i}
    \]
    and
    \[
    f_B(\phi(\prod_i \gamma_i^{j_i})) = \prod_i \gamma_i(Y_1, Y_2, Y_3)^{j_i}
    \]
    are $(\varepsilon', \psi)$-representations.
\end{lemma}
We provide the proof of this lemma in \ifx\compareversion\currentversion the Supplementary material, \else Appendix~\ref{sec:pf_lm_epspsi_rep}, \fi and give here a sketch of the proof.
\begin{proof}[Sketch of the proof.]
Using the exact certificate, we obtain equations of the form
\[
\|\sqrt{\hat Z_{\pi}} T_\pi^{\sf T}\begin{pmatrix}
    I \\ -\sqrt{3/4}\mathrm{i}I
\end{pmatrix}v_{\pi, j}(X,Y) \psi \| \leq O(\sqrt{\varepsilon}).
\]
In particular, evaluating any element of the ideal $\mathcal J$ used in the proof of Theorem~\ref{thm:unique} at $X,Y,$ and $\psi$ gives a vector of norm $O(\sqrt{\varepsilon})$. Thus the group relations defining $G$ in equation \eqref{eq:group} are approximately satisfied. Moreover, we can reduce 
\[
    f(\phi(\prod_i \gamma_i^{j_i}))f(\phi(\prod_i \gamma_i^{k_i}))\psi - f(\phi(\prod_i \gamma_i^{j_i} \prod_i \gamma_i^{k_i}))\psi
\]
with respect to $\mathcal J$ using a Gr\"obner basis to show that this has norm at most $O(\sqrt{\varepsilon})$ for both $f=f_A$ and $f=f_B$, which implies that $f_A$ and $f_B$ are $(\psi, \varepsilon)$-representations.  
\end{proof}
For $n \in \N$, denote by $0_n$ the zero vector of length $n$. Let $(\pi_1, \sigma_1, \psi_1), \dots, (\pi_4, \sigma_4, \psi_4)$ be the optimal strategies defined in the proof of Theorem~\ref{thm:unique}. Recall that $d_\pi$ is the dimension of the representation $\pi$.
\begin{theorem}\label{thm:robust}
    Suppose that $(X \otimes I, I \otimes Y, \psi)$, where $X_i \in U_{n_A}(\C)$, $Y_i \in U_{n_B}(\C)$ and $\psi \in \C^{n_An_B}$, is a feasible strategy with $\beta_q - \psi^* p_3(X \otimes I, I \otimes Y)\psi = \varepsilon$.
    Then there is a local isometry $U = U_A \otimes U_B$ and states $\phi_1, \dots, \phi_4$ such that 
    \begin{align}
    \|U \psi &- 0_{m} \oplus \bigoplus_{i=1}^4\phi_i \otimes c_i\psi_i\| \leq O(\sqrt{\varepsilon}), \label{eq:thm:1}\\
    \|U X\otimes I \psi &- 0_{m} \oplus \bigoplus_{i=1}^4\phi_i \otimes (\pi_i(X) \otimes I) c_i\psi_i \| \leq O(\sqrt{\varepsilon}),\label{eq:thm:2} \\
    \|U I\otimes Y \psi &- 0_{m} \oplus \bigoplus_{i=1}^4\phi_i \otimes (I \otimes \sigma_i(Y)) c_i\psi_i \| \leq O(\sqrt{\varepsilon}),\label{eq:thm:3}
    \end{align}
    where $\sum_i c_i^2 = 1$, $c_i \geq 0$, and $m = n_An_B(|G| - \sum_i d_{\pi_i}^2d_{\sigma_i}^2)$. 
\end{theorem}
Note that this is slightly weaker than saying that CHSH mod $3$ is a robust self-test for the maximally entangled states: even though every $\psi_i$ is maximally entangled for the optimal irreducible representations, the state $\oplus_i c_i\psi_i$ is not. In principle, we can take all optimal states $\psi_i$ to be equal, which gives a state of the form $\phi \otimes \psi_{\text{opt}}$ with $\psi_{\text{opt}}$ maximally entangled. However, because there are different optimal irreducible representations, this will not simplify equations \eqref{eq:thm:2} and \eqref{eq:thm:3}.

We provide the proof of the theorem in \ifx\compareversion\currentversion the Supplementary material, \else Appendix~\ref{sec:pf_thm_robust}, \fi and give here a sketch of the proof. The proof follows the idea of the proof of \cite[Lemma~2.4]{cui_generalization_2020}, compared to which the main differences are that we require robustness instead of exact equalities,  and that there are multiple optimal irreducible representations. 
\begin{proof}[Sketch of the proof]
    By Lemma~\ref{lem:eps_psi_rep}, $f_A$ and $f_B$ are $(\varepsilon, \psi)$-representations of $G$, so by Theorem~\ref{thm:GH}, there is a local isometry $U = U_A \otimes U_B$ such that
    \[
    \psi^*( f_A(x) \otimes f_B(y) - U_A^*\tau_A(x)U_A \otimes U_B^* \tau_B(y) U_B) \psi \leq \varepsilon 
    \]
    Then $f_A(x) \otimes f_B(y)\psi \approx \tau_A\otimes \tau_B U\psi$. We can decompose
    \[
    U \psi = \bigoplus_{\pi, \sigma} U_{\pi, \sigma} \psi
    \]
    where $U_{\pi, \sigma}\psi$ is the part of $U\psi$ that corresponds to the irreducible representations $(\pi, \sigma)$ in the decomposition of $\tau$. Using that $(X, Y, \psi)$ is $\varepsilon$-optimal, we can show that $\|U_{\pi, \sigma}\psi\|^2 \leq O(\varepsilon)$, which in turn allows us to define a state that is $O(\sqrt{\varepsilon})$-close to $U\psi$ and acts as the zero vector on the non-optimal irreducible representations in $\tau$. Normalizing this vector then gives the state of the desired form, for which the inequalities \eqref{eq:thm:1}-\eqref{eq:thm:3} hold.
\end{proof}
It is in principle possible to derive the exact constants for both Lemma~\ref{lem:eps_psi_rep} and Theorem~\ref{thm:robust}. They depend on the smallest eigenvalue of $\hat Z$, the maximum eigenvalues of pairs of non-optimal irreducible representations $(\pi, \sigma)$, and on the second largest eigenvalue of the optimal pairs $(\pi_i, \sigma_i)$. However, in the proof of Lemma~\ref{lem:eps_psi_rep}, one would need to determine the exact decomposition of 
\[
 \gamma_1^{j_1}\gamma_2^{j_2} \gamma_3^{j_3}\gamma_4^{j_4} \gamma_1^{k_1}\gamma_2^{k_2} \gamma_3^{k_3}\gamma_4^{k_4}\otimes I \psi - \gamma_1^{j_1+k_1}\gamma_2^{j_2+k_2} \gamma_3^{j_3+k_3}\gamma_4^{j_4+k_4+j_2k_1}\otimes I \psi
\]
in terms of the polynomials 
\[
T_\pi^{\sf T}\begin{pmatrix}
    I \\ -\sqrt{3/4}\mathrm{i}I
\end{pmatrix}v_{\pi, j}(X,Y) \psi
\]
and the generators of the ideal $\mathcal I$.
We reduce the polynomial using a Gr\"obner basis generated by these polynomials to check that they are approximately zero, making it difficult to keep track of the exact error terms. However, since none of these steps depends on $\varepsilon$, this does not influence the bound $O(\sqrt{\varepsilon})$.

\ifx\compareversion\currentversion
\section*{Discussion}\label{sec:concl}
\else
\section{Discussion}\label{sec:concl}
\fi

In this work we provided an exact analysis of the CHSH mod $3$ Bell inequality. 
By combining symmetry reduction, high-precision semidefinite programming, and the rounding procedure for exact SDP solutions from \cite{cohn_optimality_2024}, we obtained an exact sum-of-Hermitian-squares certificate for the maximal quantum value and confirmed the optimality of the previously proposed strategy. 
Using this certificate, we characterized all optimal strategies and showed that the inequality admits,
 up to unitary transformations and symmetries,
 a unique irreducible strategy. There are $4$ symmetry-related optimal strategies that are not unitarily equivalent, which all use a maximally entangled state. We further established a robust version of this statement: an $\varepsilon$-optimal strategy is $O(\sqrt{\varepsilon})$-close to a direct sum of optimal irreducible strategies. 
 

Several directions for future work remain open. 
A natural question is whether similar techniques can be applied to the CHSH mod $d$ inequalities for larger values of $d$. 
While the present work provides an exact analysis for the case $d=3$, the resulting sum-of-Hermitian-squares certificate is already quite involved, and its structure does not clearly suggest a general pattern that could be extended to arbitrary $d$. 
Understanding whether a more systematic structure exists for these certificates would be an important step toward analyzing higher-dimensional variants. 
Another promising direction concerns further applications of the exact rounding procedure used in this work. 
In principle, the same approach could be applied to other Bell inequalities whose optimal values are currently known only numerically through semidefinite programming relaxations. 
In particular, inequalities that are solved at the second level of the hierarchy in previous numerical studies \cite[Tables~1-3]{hrga_certifying_2024} may be good candidates: if sufficiently high-precision solutions can be obtained and the associated algebraic number fields have manageable degree, the rounding procedure may allow one to recover exact optimality certificates. 

\ifx\compareversion\currentversion
\section*{Methods}
\else
\section{Methods}\label{sec:methods}
\fi
In this section we consider methods to reduce the semidefinite program \eqref{eq:general_sdp} in size. First we give our choice of border vectors $v_n$ that lead to a hierarchy of semidefinite programs, based on so-called SOS conditional expectations. Then we use symmetry reduction techniques to block-diagonalize the positive semidefinite matrix variables. Finally, we give a transformation to a real semidefinite program and a transformation to make the semidefinite programs for CHSH mod $3$ rational.

\ifx\compareversion\currentversion
\subsection*{SOS conditional expectations}
\else
\subsection{SOS conditional expectations}\label{sec:SOS_cond_exp}
\fi
Recall that the variables $X_i, Y_j$ form a group modulo the ideal $\mathcal I$. 
Using SOS conditional expectations (see, e.g., \cite[Section~3.5]{hrga_certifying_2024}), one can show that if $p$ is a sum of squares in a group algebra, then there exists a sum of squares where the polynomials involved are polynomials using the support of $p$, rather than just any polynomials in the variables $X_i$ and $Y_j$ \cite[Proposition~3.9]{hrga_certifying_2024}.
That is, instead of a basis of $\C\langle X, Y\rangle_n/\mathcal I$, we may take the border vector $v_n$ to contain a basis of the polynomials of degree $n$ in the words in the support of $p_d -\lambda$, modulo $\mathcal I$. To further reduce the size of the vector $v_n$, we take words of degree $n$ in $X_i^kY_j^k$ with $k \leq \frac{d-1}{2}$ for $d$ odd. Then the support of $p_d$ is contained in $v_1 \cup v_1^*$, rather than in $v$. 

In general, this gives polynomials of higher degree in the original variables $X_i$ and $Y_j$ at a fixed level of the hierarchy, and does not directly correspond to a level of the standard NPA hierarchy unless the polynomial has degree $1$ and the support contains all words of degree $1$. 

\begin{remark}
    Using SOS conditional expectations, it is easy to show that the semidefinite program \eqref{eq:general_sdp} has a strictly feasible point (that is, Slater's condition is satisfied). This implies that the primal and dual semidefinite program have the same optimal objective function value, and that the minimum is attained. 
    This is essentially Corollary~3.5 from \cite{hrga_certifying_2024}.
    
    Let $G_{p_d}$ be a Hermitian matrix (not necessarily positive semidefinite) such that $p_d = -v^*G_{p_d}v \mod \mathcal I$, and take $Z = G_{p_d} + MI \succ 0$, where $M$ is a large enough constant. Let $N$ be the length of the border vector $v$, then $v^*Iv = N \mod \mathcal I$ (recall that $X^{*} = X^{-1}$ for each variable $X$), so $(\lambda=MN, Z)$ is a strictly feasible solution. 
\end{remark}
\ifx\compareversion\currentversion
\subsection*{Symmetry reduction}
\else
\subsection{Symmetry reduction}
\fi
A second size reduction comes from the symmetry of the polynomial $p_d$. These symmetries allow us to block-diagonalize the Hermitian positive semidefinite variable, and to use one constraint per basis polynomial of the space of invariants rather than one constraint per basis polynomial of the full polynomial space. 


To simplify the notation, set $V = \Span\{v_n\} \subseteq \C\langle X, Y \rangle_{n'} / \mathcal{I}$, the polynomial space the polynomials $g_j$ from our sum-of-squares decomposition lie in. Here $n$ denotes the level of our hierarchy and $n'$ is the maximum degree of a polynomial in $v_n$. 

Let $\Gamma$ be a finite group acting linearly on $\C^{2d}$, 
and let $L:\Gamma \to \GL{V}$ be the representation of $\Gamma$ on $V$ given by $L(\gamma)p(X, Y) = p(\gamma^{-1}(X, Y))$ for all $\gamma \in \Gamma$. 
In particular, we require that $V$ is $\Gamma$-invariant, which is the case with our choice of $v_n$. We wish to parameterize the $\Gamma$-invariant sum-of-squares polynomials, to find a decomposition of the $\Gamma$-invariant polynomial $p_d-\lambda$, where $\Gamma$ is the group generated by the symmetries of the polynomial $p_d$ generated by \eqref{eq:sym1}-\eqref{eq:sym3}.
Note that we do not use the symmetries generated by \eqref{eq:sym4}, since those actions change the degree of a word. This would in particular imply that the action of $\Gamma$ is not induced by an action of $\Gamma$ on $\C^{2d}$.

For the following, all that is required of $L$ and $V$ is that $(L, V)$ is a finite-dimensional representation of $\Gamma$. 

Denote by $\hat \Gamma$ the set of irreducible representations of $\Gamma$, and let $\{e_{\pi, i, j}: \pi \in \hat G, i=1, \dots, m_\pi, j=1, \dots, d_\pi\}$ be a symmetry adapted basis of $V$, where $m_\pi$ is the multiplicity of the irreducible representation $\pi$ in $L$ and $d_\pi$ is the dimension of $\pi$. That is, 
the spaces $H_{\pi,i} = \Span\{e_{\pi,i,j} : j=1, \dots, d_\pi\}$ are irreducible representations of $\Gamma$ such that $H_{\pi, i}$ is equivalent to $H_{\pi', i'}$ if and only if $\pi$ is equivalent to $\pi'$, and for each $\pi, i, i'$ there are $\Gamma$-equivariant isomorphisms $T_{\pi,i,i'}: H_{\pi,i} \to H_{\pi, i}$ such that $T_{\pi, i, i'}e_{\pi,i,j} = e_{\pi, i', j}$. Expressed in this basis the representation $L$ decomposes as
\[
L(\gamma) = \bigoplus_{\pi \in \hat \Gamma} I_{m_\pi} \otimes \pi(\gamma).
\]
\begin{proposition}
    If $p = \sum_i g_j^*g_j$ with $g_j \in V$ is $G$-invariant, then
\[
p = \sum_{\pi \in \hat \Gamma} \sum_{i, i'=1}^{m_\pi} Z_{i,i'}^{\pi} \sum_{j=1}^{d_{\pi}} e_{\pi,i,j}^* e_{\pi,i', j},
\]
where the matrices $Z^{\pi}$ are Hermitian positive semidefinite.
\end{proposition}
The proof directly translates from the commutative case (which can be found, for example, in \cite[Proposition~4.1]{leijenhorst_solving_2024-1}).

Such a symmetry adapted basis can for example be generated using the projection algorithm in \cite{serre_linear_1996}: Define the operators
\[
p_{jj'}^{(\pi)} = \frac{d_\pi}{|\Gamma|}\sum_{\gamma \in \Gamma} \pi(\gamma^{-1})_{j,j'} L(\gamma), 
\]
and choose bases $\{e_{\pi,i,1}\}$ of the image $\mathrm{Im}\left(p_{11}^{(\pi)}\right)$ of $p_{11}^{\pi}$. 
Then set $e_{\pi,i,j} = p^{(\pi)}_{j1}e_{\pi,i,1}$. 

The irreducible representations of the group we use for the symmetry reduction are constructed in \ifx\compareversion\currentversion the Supplementary material. \else Appendix~\ref{sec:irreps} \fi

\ifx\compareversion\currentversion
\subsection*{Complex to real semidefinite programs}
\else
\subsection{Complex to real semidefinite programs}
\fi
After symmetry reduction, the semidefinite program \eqref{eq:general_sdp} is complex 
with both complex constraint matrices and a complex Hermitian positive semidefinite variable matrix. To obtain a real semidefinite program, we use \cite{wang_more_2023}. 
The semidefinite program is of the form
\begin{align*}
    & \min && \lambda, \\
    & \text{subject to} && \sum_{\pi \in \hat \Gamma} \langle C^{\pi, \re}_{u} - \mathrm{i} C^{\pi,\im}_u, Z^\pi \rangle = (\lambda - p_d^{\re} - \mathrm{i}p_d^{\im})_u, && \forall u \text{ word}, \\
    &&& Z^\pi \succeq 0, && \forall \pi \in \hat G
\end{align*}
where $C^{\pi,\re}$ and $C^{\pi,\im}$ are the real and imaginary parts of the matrix  $C^{\pi} = (\sum_j e_{\pi,i,j}^*e_{\pi,i',j} \mod \mathcal I)_{i,i'}$,  and $p_u$ is the coefficient of a polynomial $p$ corresponding to a word $u$. We assume that $p$ and the entries of $C^{\pi}$ are in normal form, i.e., reduced with respect to a Gr\"obner basis of $\mathcal I$. The inner product of two complex matrices is given by $\langle A, B\rangle = \Tr(A^*B)$.
Then the real reformulation is given by
\begin{align*}
    & \min &&\lambda, \\
    & \text{subject to} && \sum_{\pi \in \hat \Gamma}\left\langle \begin{pmatrix}
        C^{\pi, \re}_u & C^{\pi,\im}_u \\
        -C^{\pi,\im}_u & C^{\pi,\re}_u
    \end{pmatrix}, Z^\pi \right\rangle = (\lambda - p_d^{\re})_u, &&\forall u \text{ word}\\
    &&& \sum_{\pi \in \hat \Gamma}\left\langle \begin{pmatrix}
        C^{\pi,\im}_u & -C^{\pi,\re}_u \\
        C^{\pi,\re}_u & C^{\pi,\im}_u
    \end{pmatrix}, Z^\pi \right\rangle = (-p_d^{\im})_u, && \forall u \text{ word}\\
    &&& Z^\pi = \begin{pmatrix}
        Z_1^\pi & (Z_2^\pi)^{\sf T} \\
        Z_2^\pi & Z_3^\pi
    \end{pmatrix} \succeq 0, && \forall \pi \in \hat \Gamma.
\end{align*}
Note in particular that there are no additional constraints on the entries of the matrix $Z^{\pi}$, such as $Z^{\pi}_1 = Z^{\pi}_3$. Given a solution $\{Z^\pi\}_{\pi}$ to the real semidefinite program, the matrices 
\[
(Z_1^{\pi} + Z^{\pi}_3) + \mathrm{i}(Z_2^{\pi} - (Z_2^{\pi})^{\sf T}) = \begin{pmatrix}
    I &\mathrm{i}I
\end{pmatrix} Z^{\pi} \begin{pmatrix}
    I \\ -\mathrm{i}I
\end{pmatrix}
\]
are a solution to the complex semidefinite program.

\ifx\compareversion\currentversion
\subsection*{Rounding and computations}\label{sec:rounding}
\else
\subsection{Rounding and computations}\label{sec:rounding}
\fi
To find an exact solution to the semidefinite program, we use the rounding procedure of \cite{cohn_optimality_2024}. This procedure gives (heuristically) an exact solution to a semidefinite program, given a sufficiently precise approximation of an optimal solution. Typically, if the exact solution is feasible and the numerical solution was (numerically) optimal, the returned solution will be optimal. However, the algorithm does not guarantee optimality. 

For the rounding procedure, one needs to give an algebraic number field such that the semidefinite program is defined over this number field and there is an optimal solution with entries in this number field. Cohn, de Laat and Leijenhorst 
provide also a heuristic in \cite{cohn_optimality_2024} to find an algebraic number field over which the optimal solution seems to be defined, but in our case, the semidefinite program is defined over a different number field. Instead of using the larger field that encompasses both number fields, we use a method to obtain a rational semidefinite program for $d=3$.

The basis elements $e_{\pi, i,j}$ have coefficients of the form $\sum_{i=0}^{d-1} c_i \omega^i$ with $c_i \in \Q$, where $\omega$ is a $d$-th root of unity, due to the irreducible representations defined in the Supplementary material. For $d = 3$, this means that the real parts of the basis elements are rational, and the imaginary parts are of the form $q\sqrt{3/4}$ with $q \in \Q$. This allows us to transform the semidefinite program to a rational semidefinite program by multiplying the matrices $Z^\pi$ from both sides by the matrices
\[
\begin{pmatrix}
    I & 0 \\
    0 & \sqrt{3/4}I
\end{pmatrix},
\]
where the identity is of the same size as the blocks $Z_i^\pi$.
Then the constraints corresponding to the real parts become rational, and the constraints corresponding to the imaginary parts will be rational after dividing by $\sqrt{3/4}$. 

If $(\{Z^\pi\}_\pi, \lambda)$ is a solution to the semidefinite program after scaling, we have
\begin{equation}
    \label{eq:SDP_constraints_final}
\sum_{\pi \in \hat \Gamma} \sum_{j=1}^{d_\pi} e_{\pi,j}^*\begin{pmatrix}
    I & \sqrt{3/4}\mathrm{i}I
\end{pmatrix} Z^\pi\begin{pmatrix}
    I \\ -\sqrt{3/4}\mathrm{i} I
\end{pmatrix} e_{\pi,j} = p_d - \lambda \mod \mathcal I
\end{equation}
where $e_{\pi, j}$ is the vector with entries $e_{\pi,i,j}$.

We implement the semidefinite program in Julia~\cite{bezanson_julia_2017}, using the high-precision solver ClusteredLowRankSolver.jl \cite{leijenhorst_solving_2024-1} and the computer algebra systems Nemo.jl and Hecke.jl \cite{fieker_nemohecke_2017}. 
Due to the reductions, the computations for the second level ($n=2$) of this hierarchy for $d = 3$  only take a few minutes even with $256$ bits of precision on a typical laptop.

\section*{Data availability}
The data generated for this paper is available at \cite{code_chshmod3}.

\section*{Code availability}
The code used in this paper is available at \cite{code_chshmod3}.

\section*{Acknowledgments} 
This work has been supported by European Union’s HORIZON-MSCA-2023-DN-JD programme under the Horizon Europe (HORIZON) Marie Skłodowska-Curie Actions, grant agreement 101120296 (TENORS), the project COMPUTE, funded within the QuantERA II Programme that has received funding from the EU's H2020 research and innovation programme under the GA No 101017733 \euflag. 
Initial computation has been performed using HPC resources from CALMIP (Grant 2023-P23035). 
IK also acknowledges support of the Slovenian Research Agency program P1-0222 and grants J1-50002, N1-0217, J1-60011, J1-50001, J1-3004 and J1-60025. 
Partially supported by the Fondation de l’\'Ecole polytechnique
as part of the Gaspard Monge Visiting Professor Program. IK  thanks \'Ecole polytechnique and Inria Paris Saclay
for hospitality during the preparation of this manuscript.

\section*{Author contributions}
I.K., N.L., and V.M. conceived the idea and prepared the paper. 
N.L. designed the code and the proofs. 

\section*{Competing interests}
The authors declare no competing interests. 

\appendix
\ifx\compareversion\currentversion
\section*{Supplementary material} 
\subsection*{Irreducible representations}\label{sec:irreps}
\else
\section{Irreducible representations}\label{sec:irreps}
\fi


Let $C_d$ be the cyclic group of order $d$. The polynomial $p_d$ is invariant under the symmetries listed in the main text. The symmetries that do not change the degree of words form the group $\Gamma =  (C_d \times C_d) \rtimes (C_2 \times C_2)$, where the first part comes from raising an index in $X$ and multiplying $Y_j$ by $\omega^j$ and vice versa, and the second part comes from the actions of interchanging $X$ and $Y$ and from negating the indices modulo $d$. 
Since inverting the matrices changes the degree of a word, we do not include that in the symmetries of $p_d$ used for the symmetry reduction. We build the irreducible representations of $\Gamma$ from the irreducible representations of $C_2$ and $C_d$ using the representation theory of finite groups \cite{serre_linear_1996}. 

Let $k \in \{0, \dots, j-1\}$, and let $\xi_j$ be a generator of $C_j$.
The irreducible representations are fully determined by their value on $\xi_j$. The group $C_j$ is abelian, so all irreducible representations are $1$-dimensional. Furthermore, for every representation $\pi$ we have $\pi(\xi_j)^j = \pi(\xi_j^j) = \pi(e) = 1$,
so $\pi(\xi_j)$ is a $j$-th root of unity. This gives the representations
\[
\pi^k(\xi_j) = \omega^k,
\]
where $\omega = \exp(2\pi \mathrm{i}/j)$. 
This gives $j = |C_j|$ non-isomorphic irreducible 1-dimensional representations, so these are all irreducible representations of $C_j$ by \cite[Corollary~2 of Proposition~5]{serre_linear_1996}. We will use $\xi_j$ for the generator of $C_j$, and $\alpha$ and $\zeta$ for general group elements.

The irreducible representations of the direct product of two groups $G_1$ and $G_2$ can be constructed from the irreducible representations of the groups themselves, using the tensor product. The tensor product of two representations $\sigma_1$ and $\sigma_2$ is defined by
\[
(\sigma_1 \otimes \sigma_2)((\zeta_1, \zeta_2)) = \sigma_1(\zeta_1) \otimes \sigma_2(\zeta_2)
\]
for $(\zeta_1, \zeta_2) \in G_1 \times G_2$.
\begin{theorem}[\protect{\cite[Theorem~10]{serre_linear_1996}}]
    If $\sigma_i$ is an irreducible representation of $G_i$ for $i=1, 2$, then $\sigma_1 \otimes \sigma_2$ is an irreducible representation of $G_1 \times G_2$. Moreover, every irreducible representation of $G_1 \times G_2$ is isomorphic to a representation $\sigma_1 \otimes \sigma_2$ where $\sigma_i$ is an irreducible representation of $G_i$.
\end{theorem}
%

The irreducible representations of the semidirect product are more complicated. Since the normal subgroup in the relevant semidirect product is abelian, it is possible to describe the irreducible representations using \cite[Section~8.2]{serre_linear_1996}. In the following, we do this for the group $\Gamma = A \rtimes H$, where $A = C_d \times C_d$ and $H = C_2 \times C_2$. We denote the irreducible representations of $A$ by $\pi^{i,j} = \pi^i \otimes \pi^j$ with $i, j=0, \dots, d-1$.  Since $A$ is abelian, these representations form a group $X = \Hom(A, \C^*)$. The product of two representations $\pi^{i,j}$ and $\pi^{k,l}$ in this group is given by
\[
(\pi^{i,j}\pi^{k,l})(\xi_d^{a_1}, \xi_d^{a_2}) = \pi^{i,j}(\xi_d^{a_1},\xi_d^{a_2}) \pi^{k,l}(\xi_d^{a_1},\xi_d^{a_2}) = \omega^{a_1(i+k)+a_2(j+l)} = \pi^{i+k,j+l}(\xi_d^{a_1}, \xi_d^{a_2}),
\]
where the indices of the representations are taken modulo $d$. The group $\Gamma$ acts on $X$ by
\[
\zeta\pi(\alpha) = \pi(\zeta^{-1} \alpha \zeta) 
\]
for $\zeta \in \Gamma, \pi \in X$ and $\alpha \in A$. 
Since $A$ is abelian, we only need to consider $\zeta \in H$.
Recall that the first group $C_2$ of the direct product interchanges the noncommutative variables $X_i$ and $Y_i$,
while the second group inverses the indices mod $d$. Then we have
\[
(\xi_2, e)\pi^{i,j}((\zeta_1, \zeta_2)) = \pi^{i,j}((\zeta_2, \zeta_1)) = \pi^{j,i}((\zeta_1, \zeta_2))
\]
and
\[
(e, \xi_2) \pi^{i,j}((\zeta_1, \zeta_2)) = \pi^{i,j}((\zeta_1^{-1}, \zeta_2^{-1})) = \pi^{d-i,d-j}((\zeta_1, \zeta_2))
\]
for $\zeta_1, \zeta_2 \in C_d$.

Let $\{\pi^{i,j}\}$ be a set of representatives of the orbits of $X/H$. Let $H_{i,j}$ be the subgroup of $H$ consisting of all elements of $H$ that fix $\pi^{i,j}$, and consider the corresponding subgroups $\Gamma_{ij} = A H_{ij}$. We extend $\pi^{i,j}$ to $\Gamma_{ij}$ by setting
\[
\pi^{i,j}(\alpha\zeta) = \pi^{i,j}(\alpha)
\]
for $\alpha \in A$ and $\zeta \in H_{ij}$. In our case, an orbit consists of the representations with indices $\{(i,j), (j,i), (d-i, d-j), (d-j, d-i)\}$. The groups $H_{ij}$ are given by
\[
H_{ij} = \begin{cases}
    \{(e,e)\} & \text{ if } i \neq j, \\
    C_2 \times \{e\} & \text{ if } i = j \text{ and } i, j \neq 0, \\
    \{(e,e), (\xi_2, \xi_2)\} & \text{ if } i+j = 0 \mod d \text{ and } i, j \neq 0, \\
    C_2 \times C_2 & \text{ if } i = j = 0.
\end{cases}
\]

Let $\rho$ be an irreducible representation of $H_{ij}$; this gives an irreducible representation on $\Gamma_{ij}$ by composition with the canonical projection $\Gamma_{ij} \to H_{ij}$.  Take $\theta_{ij, \rho}$ to be the representation induced by the tensor product $\pi^{i,j} \otimes \rho$. 
\begin{proposition}[\protect{\cite[Proposition~25]{serre_linear_1996}}]
    The representation $\theta_{ij,\rho}$ is irreducible. Moreover, each irreducible representation of $\Gamma$ is isomorphic to one of the $\theta_{ij, \rho}$, and if $\theta_{ij, \rho}$ and $\theta_{i'j', \rho'}$ are isomorphic, then $i=i'$, $j=j'$, and $\rho$ is isomorphic to $\rho'$.
\end{proposition}

An induced representation is defined as follows. Let $H$ be a subgroup of $\Gamma$, and consider the left cosets $\zeta H = \{\zeta\alpha : \alpha \in H\}$ for $\zeta \in \Gamma$. Let $\zeta_1, \dots, \zeta_m$ be representatives of the left cosets of $H$. Let $(\rho, V)$ be a representation of $H$. The induced representation of $(\rho, V)$ is the space $\bigoplus_{i=1}^m \zeta_i V$, where elements of $\zeta_i V$ are written as $\zeta_i v$ with $v \in V$. For each $\zeta_i$, and each $\zeta \in \Gamma$ there is a unique $\zeta_j$ and $\alpha_i \in H$ such that $\zeta\zeta_i = \zeta_j\alpha_i$. Since $\{\zeta_i\}_i$ is a full set of representatives of the left cosets, $j = j(i)$ is a permutation depending on $\zeta$. The action of the induced representation is then given by $\theta(\zeta) \sum_i \zeta_i v_i = \sum_i \zeta_{j(i)} \rho(\alpha_i)v_i$. 

As an example, we construct $\theta_{11, \pi^{1,0}}$. The representation $\pi^{1,1} \otimes \pi^{1,0}$ is given by
\[
\pi^{1,1} \otimes \pi^{1,0}((\xi_d^a, \xi_d^b, \xi_2^c, e)) = (-1)^c\omega^{a+b}.
\]
for $a, b\in \{0, \dots, d-1\}$ and $c\in \{0, 1\}$. The vector space is $1$-dimensional, and the representatives of the left cosets are given by $(e,e,e,\xi_2^{b_2})$. Hence the vector space for the induced representation is $2$-dimensional, where the first coordinate corresponds to $b_2 = 0$ and the second coordinate to $b_2 = 1$.

The product between a general group element  $(\xi_d^{a_1}, \xi_d^{a_2}, \xi_2^{b_1}, \xi_2^{b_2})$ and a left-coset representative $(e,e,e,\xi_2^c)$ is given by
\[
(\xi_d^{a_1}, \xi_d^{a_2}, \xi_2^{b_1}, \xi_2^{b_2})(e,e,e,\xi_2^c) = (\xi_d^{a_1}, \xi_d^{a_2}, \xi_2^{b_1}, \xi_2^{b_2+c}) = (e,e,e,\xi_2^{b_2+c}) (\xi_d^{d-a_1}, \xi_d^{d-a_2}, \xi_2^{b_1}, e).
\]
Hence $b_2 = 1$ interchanges the two subspaces, and on the second subspace the representation acts as $\pi^{11}\otimes \pi^{10}((\xi_d^{d-a_1}, \xi_d^{d-a_2}, \xi_2^{b_1}, e)) = \pi^{d-1,d-1} \otimes \pi^{10}((\xi_d^{a_1}, \xi_d^{a_2}, \xi_2^{b_1}, e))$.
That is, the induced representation is given by
\[
\theta_{11,\pi^{1,0}}((\xi_d^{a_1}, \xi_d^{a_2}, \xi_2^{b_1}, \xi_2^{b_2})) = \begin{pmatrix}
    (-1)^{b_1}\omega^{a_1+a_2} & 0 \\
    0 & (-1)^{b_1}\omega^{-a_1-a_2}
\end{pmatrix} \begin{pmatrix}0 & 1 \\ 1 & 0 \end{pmatrix}^{b_2}.
\]
This gives $k$-dimensional representations for the $H_{ij}$ corresponding to orbits of size $k$.

\ifx\compareversion\currentversion
\subsection*{Proof of Theorem~\ref{thm:equiv_flat_groebner}}
\else
\section{Proof of Theorem~\ref{thm:equiv_flat_groebner}}\label{sec:pf_thm_equiv}
\fi
\begin{theorem}[Restatement of Theorem~\ref{thm:equiv_flat_groebner}]
    Let $(\lambda, Z; M) \in \R \times \C^{N \times N} \times \C^{N \times N}$ be a primal-dual optimal solution with $\mathrm{rank} Z + \mathrm{rank} M = N$ and $\lambda = \langle G_p, M\rangle$. If $M$ is $\delta$-flat, then $H_{\mathcal J}$ is finite dimensional, and the representations defined in \ifx\compareversion\currentversion the sections on extraction through sum-of-squares certificates and flatness \else Sections~\ref{sec:grobner_optimizer} and \ref{sec:flatness} \fi are equivalent.
\end{theorem} 
\begin{proof}
    Consider $M = R_n^*R_n$ and let $\{w_a\}_{a \in U}$ be a basis of the column space of $R_{n-\delta}$ as before. For columns corresponding to $b \not\in U$, we have a unique decomposition
    \begin{equation}\label{eq:lincomb_mon}
    w_b = \sum_{a \in U} c_{a,b} w_a,
    \end{equation}
    since $M$ is $\delta$-flat.
    Let $T_{a,b} = c_{a,b}$ be a matrix in which the rows are indexed with the same words as $M$, and the columns are indexed with $b \not \in U$. Then setting $c_{b,b} = -1$ and $c_{a,b} = 0$ for distinct $a,b \not\in U$, we have $R_n T = 0$, and the columns of $T$ form a basis for the nullspace of $R_n$. Moreover, since $\rank Z + \rank M = N$, and $\langle Z, M\rangle = 0$ by optimality, the columns of $T$ form a basis of the column space of $Z$, so $Z = T \hat Z T^*$ for some positive definite $\hat Z$. Consider the ideal generated by the generators of $\mathcal I$ together with $T_i^*v_n \psi$ for each column $T_i$. Because the vectors $w_a$ for $a \in U$ form a basis of the column space of $R_{n-\delta}$, equation \eqref{eq:lincomb_mon} implies that $b\psi$ is of degree at most $n-\delta$ in the variables $X$ after reducing it by the ideal $\mathcal J$, for every word $b \in \C\langle X \rangle_n$. Additionally, $\{a\psi : a \in U\}$ is a basis of $\C\langle X \rangle_n \psi /\mathcal J$.  

    Let $\rho_M$ be the representation defined by the action of $X$ on $R_n$, and $\rho_{S}$ the representation defined by the action on $\C\langle X\rangle \psi/\mathcal J$. First note that for $b \in U$, we have
    \[
    \rho_M(X_i)w_b = w_{X_ib} = \sum_{b \in U} c_{a, X_ib} w_a,
    \]
    where in the last equality we have $c_{a,X_ib} = \delta_{a, X_ib}$ if $X_ib \in U$. That is, in the basis $\{w_a : a \in U\}$, $\rho_M(X_i)_{a,b} = c_{a, X_ib}$ for $a,b \in U$. Furthermore, by construction, we have 
    \[
    \rho_S(X_i)b\psi = X_ib\psi = \sum_{a \in U} c_{a, X_ib} a\psi \mod \mathcal J.
    \]
    That is, in the bases $\{a\psi : a \in U\}$ and $\{w_a : a \in U\}$, $\rho_S(X_i) = \rho_M(X_i)$ entrywise for all $i$. 
\end{proof}

\ifx\compareversion\currentversion
\subsection*{Proof of Lemma~\ref{lem:eps_psi_rep}}
\else
\section{Proof of Lemma~\ref{lem:eps_psi_rep}}\label{sec:pf_lm_epspsi_rep}
\fi
\begin{lemma}[Restatement of Lemma~\ref{lem:eps_psi_rep}]
    Suppose $(X \otimes I, I \otimes Y, \psi)$ is a feasible strategy with $\beta_q - \psi^* p_3(X \otimes I, I \otimes Y)\psi \leq \varepsilon$. Then there is some $\varepsilon' = O(\varepsilon)$ such that $f_A:G \to U_{n_A}(\C)$ and $f_B:G \to U_{n_B}(\C)$ defined by
    \[
    f_A(\phi(\prod_i \gamma_i^{j_i})) = \prod_i \gamma_i(X_1, \dots, X_3)^{j_i}
    \]
    and
    \[
    f_B(\phi(\prod_i \gamma_i^{j_i})) = \prod_i \gamma_i(Y_1, \dots, Y_3)^{j_i}
    \]
    are $(\varepsilon', \psi)$-representations.
\end{lemma}
\begin{proof}
In the following, we write $p(X, Y)$ instead of $p(X \otimes I, I \otimes Y)$ when evaluating a polynomial $p$ on the strategy for notational simplicity.
Let $(X \otimes I, I \otimes Y, \psi)$ be a strategy satisfying $X_i^3 = I$ and $Y_j^3 = I$, with $\beta_{q} -\psi^*p_3(X, Y)\psi = O(\varepsilon)$.  
Then using the sum-of-squares decomposition we have
\[
\psi^* \sum_{\pi, j} v_{\pi, j}^*(X, Y)\begin{pmatrix}
    I & \sqrt{3/4}\mathrm{i}I
\end{pmatrix}T_{\pi} \hat Z_\pi T_{\pi}^{\sf T} \begin{pmatrix}
    I \\ -\sqrt{3/4}\mathrm{i}I
\end{pmatrix} v_{\pi, j}(X, Y) \psi = O(\varepsilon),
\]
so in particular
\[
\|\sqrt{\hat Z_{\pi}} T_\pi^{\sf T}\begin{pmatrix}
    I \\ -\sqrt{3/4}\mathrm{i}I
\end{pmatrix}v_{\pi, j}(X,Y) \psi \| = O(\sqrt{\varepsilon})
\]
for every $\pi, j$.
Since $\hat Z$ is fixed, the elements in $\mathcal J$ evaluated at $X, Y$ have norm $O(\sqrt{\varepsilon})$. 
In particular, the following approximate relations hold with an error of $O(\sqrt{\varepsilon})$ for the matrices $\gamma_i =\gamma_i(X_1, \dots, X_3)$: 
\begin{enumerate}
    \item $uv\otimes I\psi\approx vu\otimes I\psi$ with $v = \gamma_4^{k}$ and $u = \gamma_1^{i_1}\gamma_2^{i_2}\gamma_3^{i_3}$,
    \item $uv\gamma_4^{i_4}\otimes I\psi\approx vu\gamma_4^{i_4}\otimes I\psi$ with $v = \gamma_3^{k}$ and $u = \gamma_1^{i_1}\gamma_2^{i_2}$,
    \item $\gamma_1^{i_1}\gamma_2^{k}\gamma_3^{i_3}\gamma_4^{i_4+i_1k}\otimes I\psi \approx \gamma_2^{k}\gamma_1^{i_1}\gamma_3^{i_3}\gamma_4^{i_4}\otimes I\psi$, 
    \item $\gamma_j^{3}\prod_{l=j+1}^4\gamma_l^{i_l}\otimes I \psi \approx \prod_{l=j+1}^4\gamma_l^{i_l}\psi$,
\end{enumerate}
for $i \in \{0,1,2\}^4$ and $k \in \{1, \dots, 4\}$; the code to verify that these $f_1-f_2 \in \mathcal J$ for the approximate equations $f_1 \approx f_2$ above is available at \cite{code_chshmod3}.
We will use these approximate relations to show that
\[
(\prod_i \gamma_i^{j_i} \prod_i \gamma_i^{k_i} \otimes I) \psi \approx (\gamma_1^{j_1+k_1} \gamma_2^{j_2+k_2}\gamma_3^{j_3+k_3}\gamma_4^{j_4+k_4+j_2k_1})\otimes I\psi,
\]
where the powers are modulo $d$, with a difference of norm $O(\sqrt{\varepsilon})$. We have
\begin{equation}\label{eq:approx_equals}
\begin{aligned}
&\gamma_1^{j_1}\gamma_2^{j_2} \gamma_3^{j_3}\gamma_4^{j_4} \gamma_1^{k_1}\gamma_2^{k_2} \gamma_3^{k_3}\gamma_4^{k_4}\otimes I \psi \\
&\approx \gamma_1^{j_1}\gamma_2^{j_2} \gamma_3^{j_3} \gamma_1^{k_1}\gamma_2^{k_2} \gamma_3^{k_3}\gamma_4^{j_4+k_4}\otimes I \psi \\
&\approx \gamma_1^{j_1}\gamma_2^{j_2} \gamma_1^{k_1}\gamma_2^{k_2} \gamma_3^{j_3+k_3}\gamma_4^{j_4+k_4}\otimes I \psi \\
&\approx \gamma_1^{j_1}\gamma_2^{j_2} \gamma_1^{k_1}\gamma_2^{k_2} \gamma_3^{j_3+k_3}\gamma_4^{j_4+k_4}\otimes I \psi\\
&\approx \gamma_1^{j_1}\gamma_2^{j_2+k_2}\gamma_1^{k_1} \gamma_3^{j_3+k_3}\gamma_4^{j_4+k_4 + 2k_1k_2}\otimes I \psi \\
&\approx \gamma_1^{j_1+k_1}\gamma_2^{j_2+k_2} \gamma_3^{j_3+k_3}\gamma_4^{j_4+k_4+3k_1k_2+j_2k_1}\otimes I \psi \\
&\approx \gamma_1^{j_1+k_1}\gamma_2^{j_2+k_2} \gamma_3^{j_3+k_3}\gamma_4^{j_4+k_4+j_2k_1}\otimes I \psi,
\end{aligned}
\end{equation}
where each time we first move a term $\gamma_i^{k_i}$ to the term $\gamma_i^{j_i}$ at the front, then move the resulting product $\gamma_i^{j_i+k_i}$ to the appropriate place at the back together, and finally reduce the powers modulo $3$ (although for simplicity this is not shown in the equations). 
Since the group elements $\gamma_1$ and $\gamma_2$ do not commute, the steps where we interchange the corresponding matrices are displayed more carefully above. 

This shows that $f_A:G \to U_{n_A}$ defined by $f(\phi^{-1}(\prod_i\gamma_i^{j_i})) = \prod_i \gamma_i(X)^{j_i}$, where $\gamma_i(X)$ is defined by \eqref{eq:Xi_to_fi}, is an $(\varepsilon', \psi)$-representation for some $\varepsilon' = O(\varepsilon)$. Similarly, it can be shown that $f_B$ is also an $(\varepsilon'', \psi)$-representation for some $\varepsilon'' = O(\varepsilon)$. 
\end{proof}

\ifx\compareversion\currentversion
\subsection*{Proof of Theorem~\ref{thm:robust}}
\else
\section{Proof of Theorem~\ref{thm:robust}}\label{sec:pf_thm_robust}
\fi
\begin{theorem}[Restatement of Theorem~\ref{thm:robust}]
    Suppose that $(X \otimes I, I \otimes Y, \psi)$, where $X_i \in U_{n_A}(\C)$, $Y_i \in U_{n_B}(\C)$ and $\psi \in \C^{n_An_B}$, is a feasible strategy with $\beta_q - \psi^* p_3(X \otimes I, I \otimes Y)\psi = \varepsilon$.
    Then there is a local isometry $U = U_A \otimes U_B$ and states $\phi_1, \dots, \phi_4$ such that 
    \begin{align}
    \|U \psi &- 0_{m} \oplus \bigoplus_{i=1}^4\phi_i \otimes c_i\psi_i\| \leq O(\sqrt{\varepsilon}), \label{eq:thm:1bis}\\
    \|U X\otimes I \psi &- 0_{m} \oplus \bigoplus_{i=1}^4\phi_i \otimes (\pi_i(X) \otimes I) c_i\psi_i \| \leq O(\sqrt{\varepsilon}),\label{eq:thm:2bis} \\
    \|U I\otimes Y \psi &- 0_{m} \oplus \bigoplus_{i=1}^4\phi_i \otimes (I \otimes \sigma_i(Y)) c_i\psi_i \| \leq O(\sqrt{\varepsilon}),\label{eq:thm:3bis}
    \end{align}
    where $\sum_i c_i^2 = 1$, $c_i \geq 0$, and $m = n_An_B(|G| - \sum_i d_{\sigma_i}^2d_{\rho_i}^2)$
\end{theorem}
\begin{proof}
    Let $(X \otimes I, I \otimes Y, \psi)$ be as in the theorem statement. As before, we write $p(X, Y)$ instead of $p(X \otimes I, I \otimes Y)$.
    
    By Lemma~\ref{lem:eps_psi_rep}, both $f_A$ and $f_B$ are $(\varepsilon, \psi)$-representations of $G$, so by Theorem~\ref{thm:GH} there is a local isometry $U = U_A \otimes U_B$ with 
    \[
    \psi^*( f_A(x) \otimes f_B(y) - U_A^*\tau_A(x)U_A \otimes U_B^* \tau_B(y) U_B) \psi \leq \varepsilon 
    \]
    for all $x, y \in G$. Recall that we can write $\tau_A = \bigoplus_\pi I_{n_A} \otimes I_{d_\pi} \otimes \pi$ and $\tau_B = \bigoplus_\sigma I_{n_B} \otimes  I_{d_\sigma} \otimes \sigma$, where the sums run over all irreducible representations of $G$.

    We can decompose $U_A$ and $U_B$ into parts for each irreducible representation $\pi$, so that
    \[
    U_Au = \bigoplus_\pi U_{A, \pi}u,\ U_Bu = \bigoplus_\sigma U_{B, \sigma} u
    \]
    for $u \in \mathcal H_A$ and $u \in \mathcal H_B$ respectively.
    Now define 
    \begin{equation}
        \label{eq:c_pi_sigma}
    c_{\pi, \sigma} = \|U_{A, \pi} \otimes U_{B, \sigma} \psi\|^2
    \end{equation}
    and the normalized states
    \begin{equation}
        \label{eq:psi_hat}
    \hat\psi_{\pi, \sigma} = \begin{cases}
    \frac{1}{\sqrt{c_{\pi,\sigma}}}U_{A, \pi} \otimes U_{B, \sigma} \psi & \text{ if } c_{\pi, \sigma} > 0, \\
    0 & \text{ if } c_{\pi, \sigma} = 0.
    \end{cases}
    \end{equation}
    Note that $\sum_{\pi, \sigma}c_{\pi, \sigma} = 1$. Set $\hat \psi = U\psi$. 
    Then for the strategy $(\tau_A(X), \tau_B(Y), \hat \psi)$ we have
    \[
     \hat\psi^* p_3(\tau_A(X), \tau_B(Y)) \hat\psi = \sum_{\pi, \sigma}c_{\pi, \sigma} \hat\psi_{\pi,\sigma}^*( p_3(I_{n_Ad_{\pi}} \otimes \pi(X),  I_{n_Bd_{\sigma}} \otimes \sigma(X))) \hat\psi_{\pi, \sigma}, 
    \]
    which is a convex combination of the values from using the strategies $(I \otimes \pi, I\otimes \sigma, \hat\psi_{\pi, \sigma})$.

    \begin{lemma}\label{lem:estimate_ci}
    Let $(\pi_i, \sigma_i, \psi_i)$ be the optimal irreducible strategies for $p_3$, and $c_{\pi_i, \sigma_i}$ as defined in \eqref{eq:c_pi_sigma}. Then
    \begin{equation*}
    \sum_i c_{\pi_i, \sigma_i} \geq 1-O(\varepsilon).
    \end{equation*}
    \end{lemma}
        

    \begin{lemma}\label{lem:psi_dif}
    Let $(\pi_i, \sigma_i, \psi_i)$ be optimal irreducible strategies for $p_3$, and $c_{\pi_i, \sigma_i}$ and $\hat \psi_{\pi_i, \sigma_i}$ as defined in \eqref{eq:c_pi_sigma} and \eqref{eq:psi_hat}. Then there is some state $\phi_i$ such that
\begin{equation*}
    c_{\pi_i, \sigma_i}\|\hat \psi_{\pi_i, \sigma_i} - \phi_i \otimes \psi_{i}\|^2 \leq O(\varepsilon).
    \end{equation*}
    \end{lemma}
    We postpone the proofs of these lemmas until after the proof of the theorem.
    Now consider the state 
    \[
    0_{m} \oplus\bigoplus_{i}\phi_i\otimes \sqrt{c_i}\psi_{i},
    \]
    where $m=n_An_B\sum_{(\pi, \sigma) \neq (\pi_i, \sigma_i)}d_{\pi}^2d_{\sigma}^2 = n_An_B(|G| - \sum_i d_{\pi_i}^2d_{\sigma_i}^2)$, and $c_i = c_{\pi_i, \sigma_i}/(\sum_{i'} c_{\pi_{i'}, \sigma_{i'}})$. Note that it is indeed a unit vector,  $c_i \geq c_{\pi_i, \sigma_i}$, and 
    \begin{equation}
        \label{eq:c_dif}
    \sum_i |c_{\pi_i, \sigma_i} - c_i| = \sum_i c_{\pi_i, \sigma_i}(\frac{1}{\sum_{i'} c_{\pi_{i'}, \sigma_{i'}}} - 1) =  1-\sum_i c_{\pi_i, \sigma_i} \leq O(\varepsilon).
    \end{equation}
    
    Then we have
    \begin{align*}
    \| &\hat \psi - 0_m \oplus\bigoplus_{i}\phi_i\otimes \sqrt{c_i}\psi_{i}\|^2 \\
    &\leq \sum_{(\pi, \sigma) \neq (\pi_i, \sigma_i)} c_{\pi, \sigma}\|\hat \psi_{\pi, \sigma} \|^2 + \sum_i \left(c_{\pi_i, \sigma_i}\|\hat \psi_{\pi_i, \sigma_i} - \phi_i \otimes \psi_{i} \|^2 + |c_{\pi_i, \sigma_i} - c_i| \|\phi_i \otimes \psi_i\|^2\right) \\
    &\leq O(\varepsilon) + O(\varepsilon) + O(\varepsilon)
    \end{align*}
    by Lemma~\ref{lem:estimate_ci}, Lemma~\ref{lem:psi_dif} and equation~\eqref{eq:c_dif}. This proves inequality \eqref{eq:thm:1bis}.

    Next, we consider the action of an operator $X$ on $\psi$. We have
    \[
    \|X \otimes U_B \psi - U_A^*\tau_A(X)U_A \otimes U_B\psi\|^2 \leq O(\varepsilon)
    \]
    from Theorem~\ref{thm:GH}. Multiplying both terms by $U_A \otimes I$ gives
    \begin{equation}
        \label{eq:bound_UX}
    \|U_AX \otimes U_B \psi - U_AU_A^*\tau_A(X)U_A \otimes U_B\psi\|^2 \leq O(\varepsilon),
    \end{equation}
    and since $U_AU_A^*$ is a projection onto the column space of $U_A$, and $\tau_A(X)$ acts on the column space of $U_A$, we have
    \[
    U_AU_A^*\tau_A(X)U_A\otimes U_B\psi = \tau_A(X)U_A\otimes U_B\psi = \bigoplus_{\pi, \sigma} I_{n_Ad_\pi} \otimes \pi(X) \otimes I_B \sqrt{c_{\pi, \sigma}}\hat\psi_{\pi, \sigma}.
    \]
    
     Furthermore, 
     \begin{equation}
         \label{eq:dif_pipsi}
    \begin{aligned}
    \|&\bigoplus_{\pi, \sigma} I_{n_Ad_\pi} \otimes \pi(X) \otimes I_B \sqrt{c_{\pi, \sigma}}\hat\psi_{\pi, \sigma} - 0_m \oplus \bigoplus_{i=1}^4\phi_i \otimes (\pi_i(X) \otimes I) \sqrt{c_i}\psi_i\|^2 \\
    &= \sum_{(\pi, \sigma) \neq (\pi_i, \sigma_i)} c_{\pi, \sigma} \\
    & \quad + \sum_i \| I_{n_Ad_{\pi_i}} \otimes \pi_i(X) \otimes I_B \sqrt{c_{\pi_i, \sigma_i}}\hat\psi_{\pi_i, \sigma_i}  - (I \otimes \pi_i(X) \otimes I_{d_{\sigma_i}})\sqrt{c_i}\phi_i \otimes \psi_i\|^2 \\
    &\leq O(\varepsilon) + \sum_i\big( c_{\pi_i, \sigma_i}\| (I_{n_Ad_{\pi_i}} \otimes\pi_i(X) \otimes I_B)(\hat\psi_{\pi_i, \sigma_i} - \phi_i \otimes \psi_i)\|^2 \\
    & \quad +  |c_{\pi_i, \sigma_i} - c_i| \|(I_{n_A d_{\pi_i}} \otimes \pi_i(X) \otimes I_{d_{\sigma_i}} ) \phi_i \otimes \psi_i\|^2\big) \\
    &\leq O(\varepsilon) + O(\varepsilon)  + O(\varepsilon),
    \end{aligned} 
     \end{equation}
    where we used the triangle inequality, Lemma~\ref{lem:estimate_ci} and \ref{lem:psi_dif}, and equation \eqref{eq:c_dif}. 
    Using both \eqref{eq:bound_UX} and \eqref{eq:dif_pipsi} gives 
    \begin{align*}
    \|&U X\otimes I \psi - 0_m \oplus  \bigoplus_{i=1}^4 \phi_i \otimes (\pi_i(X) \otimes I) \sqrt{c_i}\psi_i \|^2 \\
    &\leq \|U X\otimes I \psi -  \bigoplus_{\pi, \sigma} I_{n_Ad_\pi} \otimes \pi(X) \otimes I_B \sqrt{c_{\pi, \sigma}}\hat\psi_{\pi, \sigma}\|^2 \\
    &+ \|\bigoplus_{\pi, \sigma} I_{n_Ad_\pi} \otimes \pi(X) \otimes I_B \sqrt{c_{\pi, \sigma}}\hat\psi_{\pi, \sigma} - 0_m \oplus \bigoplus_{i=1}^4 \phi_i \otimes (\pi_i(X) \otimes I) \sqrt{c_i}\psi_i\|^2 \leq O(\varepsilon),
    \end{align*}
    which is the desired inequality \eqref{eq:thm:2bis}. 
    
    The inequality \eqref{eq:thm:3bis} for applying an operator $Y$ can be derived similarly.
\end{proof}


    \begin{proof}[Proof of Lemma~\ref{lem:estimate_ci}]
        
    Let $\beta'$ be the maximum of $\lambda_{\max}(p_3(\pi, \sigma))$ with $\pi, \sigma$ irreducible such that $(\pi, \sigma) \neq (\pi_i, \sigma_i)$ for any $i$.  Then 
    \[
    \beta_q - \varepsilon = \hat \psi^* p_3(\tau_A(X), \tau_B(Y))\hat \psi \leq \sum_i c_{\pi_i, \sigma_i}\beta_q + \sum_{(\pi, \sigma) \neq (\pi_i, \sigma_i)}c_{\pi, \sigma} \beta'.
    \]
    Since $\sum_{\pi, \sigma} c_{\pi, \sigma} = 1$, this gives
    \[
    \sum_i c_{\pi_i, \sigma_i} \geq  1 - \varepsilon/ (\beta_q - \beta') = 1-O(\varepsilon).\qedhere
    \] 
    \end{proof}

    \begin{proof}[Proof of Lemma~\ref{lem:psi_dif}]
    Let $\phi \otimes \sum_k a_k\psi^k$ be the decomposition of $\hat \psi_{\pi_i, \sigma_i}$ into eigenvectors of $p_3(\pi_i, \sigma_i)$, where $\beta_q=\lambda_1 \geq \dots\geq \lambda_{d_{\pi_i}d_{\sigma_i}}$ are the eigenvalues of $p_3(\pi_i, \sigma_i)$ corresponding to eigenvectors $\psi^1, \dots, \psi^{d_{\pi_i}d_{\sigma_i}}$. Since $\beta_q - \psi^*p_3(X, Y)\psi \leq \varepsilon$, we have
    \begin{align*}
    \varepsilon &= \beta_q - \hat \psi^* p_3(\tau_A, \tau_B) \hat \psi \\
    &= \beta_q - \sum_{\pi, \sigma} c_{\pi, \sigma} \hat\psi_{\pi, \sigma}^*p_3(\pi, \sigma)\hat\psi_{\pi, \sigma} \\
    &= \beta_q(1-\sum_i c_{\pi_i, \sigma_i}) + \sum_{i}c_{\pi_i, \sigma_i}( \psi_{i}^* p_3(\pi, \sigma) \psi_{i} - \hat \psi_{\pi_i, \sigma_i}^*p_3(\pi_i, \sigma_i)\hat\psi_{\pi_i,\sigma_i}) \\
    & \phantom{{}={}}- \sum_{(\pi, \sigma)\neq (\pi_i, \sigma_i)} c_{\pi, \sigma}\hat \psi_{\pi, \sigma}^*p_3(\pi, \sigma)\hat\psi_{\pi,\sigma}.
    \end{align*}
    Using the eigendecomposition, we get
    \begin{align*}
    ( \psi_{i}^* p_3(\pi, \sigma) \psi_{i} - \hat \psi_{\pi_i, \sigma_i}^*p_3(\pi_i, \sigma_i)\hat\psi_{\pi_i,\sigma_i}) &= \lambda_1 - \sum_k \lambda_k a_k^2 \\
    &\geq \lambda_1(1-a_1^2) - \lambda_2\sum_{k=2}^{d_{\pi_i}d_{\sigma_i}} a_k^2 \\
    &= (\lambda_1 - \lambda_2)(1-a_1^2) \\
    &\geq (\lambda_1 - \lambda_2)(1-a_1),
    \end{align*}
    since $\sum_k a_k^2 = 1$ and $x^2 \leq x$ for $x \in [0, 1]$. Note that 
    \begin{align*}
    a_1 = (\phi\otimes\psi_{i})^*\hat \psi_{\pi_i, \sigma_i} = 1-\frac{1}{2}\|\hat \psi_{\pi_i, \sigma_i} - \phi \otimes \psi_{i}\|^2.
    \end{align*}
    Together, this gives
    \begin{align*}
   \sum_{i} &c_{\pi_i, \sigma_i}\frac{\beta_q -\lambda_2^i}{2}\|\hat \psi_{\pi_i, \sigma_i} - \phi_i \otimes \psi_{i}\|^2 \\
   &\leq \sum_{i} c_{\pi_i, \sigma_i} ( \psi_{i}^* p_3(\pi, \sigma) \psi_{i} - \hat \psi_{\pi_i, \sigma_i}^*p_3(\pi_i, \sigma_i)\hat\psi_{\pi_i,\sigma_i}) \\
   &= \varepsilon - \beta_q (1-\sum_i c_{\pi_i, \sigma_i}) +\sum_{(\pi, \sigma)\neq (\pi_i, \sigma_i)} c_{\pi, \sigma}\hat \psi_{\pi, \sigma}^*p_3(\pi, \sigma)\hat\psi_{\pi,\sigma} \\
   &\leq \varepsilon + \frac{\beta'}{\beta_q-\beta'}\varepsilon = O(\varepsilon)
    \end{align*}
    by Lemma~\ref{lem:estimate_ci}.
    In particular, for each $i$, we have
    \[
    c_{\pi_i, \sigma_i} \|\hat \psi_{\pi_i, \sigma_i} - \phi_i \otimes \psi_{i}\|^2 \leq O(\varepsilon). \qedhere
    \]
    \end{proof}

\newcommand{\etalchar}[1]{$^{#1}$}

\end{document}